\documentclass[11pt]{amsart}
\usepackage{amssymb, amstext, amscd, amsmath, amssymb}
\usepackage{mathtools, paralist, color, dsfont, rotating}
\usepackage{verbatim}
\usepackage{enumerate,enumitem}
\usepackage{tikz-cd}
\usepackage{float}
\usepackage{setspace}
\usepackage[textwidth=0.7in]{todonotes}
\makeatletter 
\@mparswitchfalse%
\makeatother
\normalmarginpar

\usepackage{hyperref}
\usepackage{tikz}

\floatstyle{boxed} 
\restylefloat{figure}

\numberwithin{equation}{section}
\usepackage[a4paper]{geometry}
\usepackage{quoting}
\quotingsetup{vskip=.1in}
\quotingsetup{leftmargin=.17in}
\quotingsetup{rightmargin=.17in}
\let\OLDthebibliography\thebibliography
\renewcommand\thebibliography[1]{
  \OLDthebibliography{#1}
  \setlength{\parskip}{0pt}
  \setlength{\itemsep}{2pt plus 0.5ex}
}

\makeatletter
\def\@cite#1#2{{\m@th\upshape\bfseries%
[{#1\if@tempswa{\m@th\upshape\mdseries, #2}\fi}]}}
\makeatother
\theoremstyle{plain}
\newtheorem{theorem}{Theorem}[section]
\newtheorem{corollary}[theorem]{Corollary}
\newtheorem{proposition}[theorem]{Proposition}
\newtheorem{lemma}[theorem]{Lemma}
\theoremstyle{definition}

\newtheorem{example}[theorem]{Example}

\newtheorem{remark}[theorem]{Remark}

\newtheorem{question}[theorem]{Question}

\theoremstyle{remark}
\newcommand{\bbC}{{\mathbb{C}}}
\newcommand{\bbD}{{\mathbb{D}}}

\newcommand{\bbT}{{\mathbb{T}}}

\newcommand{\A}{{\mathcal{A}}}
\newcommand{\B}{{\mathcal{B}}}
\newcommand{\C}{{\mathcal{C}}}

\newcommand{\I}{{\mathcal{I}}}

\renewcommand{\phi}{\varphi}
\newcommand{\upchi}{{\raise.35ex\hbox{\ensuremath{\chi}}}}

\newcommand{\id}{{\operatorname{id}}}

\newcommand\diag{\mathop{\rm diag}}

\newcommand{\cmax}{\mathrm{C}^*_{\text{max}}}

\newcommand{\cstarlattice}{\text{C$^*$-Lat}}
\newcommand{\Cmax}{C^\ast_\text{max}}

\newcommand{\Ce}{C^\ast_\text{e}}

\begin{document}
%%%%%%%%%%%%%%%%%%%%%%%%%%%%%%%%%%%%%%
\title{A C*-cover lattice dichotomy}

\author{Adam~Humeniuk}
\address{Department of Mathematics and Computing, Mount Royal University, Calgary, AB, Canada}
\email{ahumeniuk@mtroyal.ca}

\author{Christopher~Ramsey}
\address {Department of Mathematics and Statistics, MacEwan University, Edmonton, AB, Canada}
\email{ramseyc5@macewan.ca}

\author{Marcel~Scherer}
\address{Fachrichtung Mathematik, Universit\"at des Saarlandes, 66123 Saarbr\"ucken, Germany}
\email{scherer@math.uni-sb.de}

%%%%%%%%%%%%%%%%
\begin{abstract}
In this paper, we show that the lattice of C*-covers of a non-selfadjoint operator algebra is either one point or uncountable. We prove that there are non-selfadjoint operator algebras with a one-point lattice in two ways: as an explicit subalgebra of the C*-algebra of a universal contraction, and via a direct limit construction inspired by the work of Kirchberg and Wassermann for operator systems. We also establish that the C*-envelope need not have an immediate successor C*-cover in the lattice, and that a semi-Dirichlet non-selfadjoint operator algebra never has a one-point lattice.
\end{abstract}
%%%%%%%%%%%%%%%%


\maketitle

\section{Introduction}

What do the self-adjoint structures that a non-selfadjoint operator algebra generates tell us about the original operator algebra? It is known that these C*-algebras form a complete lattice under a natural equivalence and so we explore the most basic forms this collection could take. Is there a non-selfadjoint operator algebra that generates a unique C*-algebra or is that a test of self-adjointness? A finite lattice? A totally ordered lattice?

Given an operator algebra $A$, a C*-algebra is called a \textit{C*-cover} if it is generated by the image of a completely isometric representation of $A$. Two such covers are equivalent if there is a *-isomorphism between the C*-algebras that is compatible with the representations of $A$. The collection of all C*-covers is called the \textit{lattice of C*-covers}, is denoted $\cstarlattice(A)$, and is a complete lattice under a natural partial order. This implies that there is a \textit{minimal} $C^*$-cover and a \textit{maximal} $C^*$-cover, historically the most significant self-adjoint structures of an operator algebra. The study and use of the lattice of C*-covers has arisen recently, see \cite{KatRamMem, Thompson, Hamidi, HumRam, HKR, HumRamThom}, and we firmly believe that it will develop into a valuable theory based on the fact that there are so many tractable open questions.

We will recall for the reader the relevant definitions, theorems, and history of C*-covers in Section \ref{Section:Background}. We also carefully work out the theory of direct limits of operator algebras and their relation to C*-covers.
After this, we investigate the cardinality of $\cstarlattice(A)$ and address Question 3.1 in \cite{HumRam}: Is there a non-selfadjoint operator algebra with only one equivalence class of $C^*$-covers? We answer this question affirmatively in Section \ref{Section:non-trivial one point lattice example} by proving the following theorem.
\vskip 6 pt
\noindent {\bf Theorem \ref{thm:onepointlattice}.} \textit{
    Let $x$ be a universal contraction. Then the operator algebra generated by $1, x, (x^*)^2, (x^*)^3$ is non-selfadjoint and has only, up to equivalence, one $C^*$-cover.
}
\vskip 6 pt
Furthermore, starting from an arbitrary non-selfadjoint operator algebra, we construct an injective inductive sequence of operator algebras whose inductive limit is non-selfadjoint and has only one equivalence class of $C^*$-covers. The construction, inspired by the Kirchberg-Wassermann operator system (see \cite[Proposition 16]{KirWas}), proceeds as follows: we consider an operator algebra $A$, embedded in its maximal $C^*$-cover, and add its Shilov ideal $I$. Repeating this step yields the inductive sequence. Motivated by the use of the object $A+I$, we study in Section \ref{Section:Extension of an operator algebra by a Shilov ideal} the extension of operator algebras by a Shilov ideal, showing that $\cstarlattice(A)$ and $\cstarlattice(A+I)$ are lattice-isomorphic. We also characterize the maximal $C^*$-cover of $A+I$ in terms of that of $A$.\\
Encouraged by the existence of non-selfadjoint operator algebras with a one-point $C^*$-cover lattice, we continue the investigation of the $C^*$-cover lattice with the question: Is there an operator algebra $A$ such that $|\cstarlattice(A)|<\infty$? Combining a classical result of Sarason and Katsnelson, we obtain our main result of this paper in Section \ref{Section:Uncountability of the C*-cover lattice}.
\vskip 6 pt
\noindent {\bf Theorem \ref{thm:maintheorem}.} \textit{
  Let $A$ be a separable operator algebra. Then
    \[
      |C^*\textup{-Lat}(A)|=\begin{cases} 1 & or\\ \mathfrak{c}, & \\ \end{cases}
  \]
  where $\mathfrak{c}$ is the continuum.
}
\vskip 6 pt
For non-separable operator algebras, the cardinality is either 1 or at least $\mathfrak{c}$. Together with the results in \cite{HumRam}, this shows that non-selfadjoint operator algebras contained in finite-dimensional $C^*$-algebras, as well as hyperrigid operator algebras, have at least $\mathfrak{c}$ distinct equivalence classes of $C^*$-covers.\\
While this result rules out many possible structures of $\cstarlattice(A)$, it remains unknown whether the lattice can be totally ordered. The proof of the lattice dichotomy shows that if there exists an immediate successor to the minimal $C^*$-cover, $\cstarlattice(A)$ cannot be totally ordered. In Section \ref{Section:No successor to envelope}, we give an example of an operator algebra for which the minimal $C^*$-cover has no immediate successor, using a result of Gramsch and Luft \cite{Gramsch, Luft} that characterizes the ideals in $\mathcal{B}(H)$ for Hilbert space $H$ of arbitrary dimension.

Finally, we address Question 3.13 of \cite{HKR} by proving:.
\vskip 6 pt
\noindent {\bf Theorem \ref{thm:semidirichlet}.} \textit{
   If $A$ is a non-selfadjoint, semi-Dirichlet operator algebra, then the maximal C*-cover is not a semi-Dirichlet C*-cover.
}
\vskip 6 pt
Together with the lattice dichotomy, this implies that non-selfadjoint semi-Dirichlet operator algebras possess infinitely many different equivalence classes of $C^*$-covers.

\section{Background}\label{Section:Background}

\subsection{C*-covers}

Given an operator algebra $A$, a pair $(\mathcal{A},j)$, consisting of a $C^*$-algebra $\mathcal{A}$ and a completely isometric map $j:A\to\mathcal{A}$ such that $\A = j(A)$, is called a \textit{$C^*$-cover} of $A$. Given two $C^*$-covers $(\mathcal{A}_1,j_1), (\mathcal{A}_2,j_2)$ of $A$ a $*$-homomorphism $\Phi:\mathcal{A}_1\to\mathcal{A}_2$ such that $\Phi j_1=j_2$ is called a \textit{morphism of C*-covers} and is denoted
\[
\Phi : (\mathcal{A}_1,j_1) \rightarrow (\mathcal{A}_2,j_2)\,. 
\]
To the best of our knowledge, C*-covers and their morphisms were first described in \cite[Chapter 2]{Blecher}.
This morphism, when it exists, is in fact unique and thus we simply write 
\[
(\mathcal{A}_2,j_2)\preceq (\mathcal{A}_1,j_1)\,
\]
if such a morphism exists. Uniqueness also implies that if $(\mathcal{A}_2,j_2)\preceq (\mathcal{A}_1,j_1)$ and $(\mathcal{A}_1,j_1)\preceq (\mathcal{A}_2,j_2)$ then there is a $*$-isomorphism $\Phi:\A_1 \rightarrow \A_2$ satisfying $\Phi j_1 = j_2$. In this case, we write 
  \[
    (\mathcal{A}_1,j_1)\sim(\mathcal{A}_2,j_2)\,.
 \]
 
It is straightforward to verify that $\sim$ defines an equivalence relation. Denote by $[\mathcal{A},j]$ the equivalence class of all $C^*$-covers equivalent to $(\mathcal{A},j)$. Then, the \textit{lattice of $C^*$-covers},
  \[
    \cstarlattice(A)=\{[\mathcal{A},j] \ | \ (\mathcal{A},j)\ \textup{$C^*$-cover of A}\},
  \]
forms a set with an inherited partial order, which we denote with the same symbol $\preceq$. For a family $(f_i)_{i\in I}$ in $\cstarlattice(A)$, an element $f\in\cstarlattice(A)$ is called the \textit{join} of $(f_i)_{i\in I}$ if 
  \begin{enumerate}
      \item $f_i\preceq f$ for all $i\in I$,
      \item if $g\in\cstarlattice(A)$ satisfies $f_i\preceq g$ for all $i\in I$, then $f\preceq g$.
  \end{enumerate}
It is straightforward to verify that if such an element $f$ exists, then it is unique. Another term for the join is \textit{least upper bound}. In particular, the join of $\{[\A_i, j_i]\}_{i\in I}$, a collection of C*-covers of $A$, is $[C^*(j(A)), j]$ where
\[
j = \bigoplus_{i\in I} j_i\,.
\]\\
Dually, an element $f\in\cstarlattice(A)$ is called the \textit{meet} of $(f_i)_{i\in I}$ if 
  \begin{enumerate}
      \item $f\preceq f_i$ for all $i\in I$,
      \item if $g\in\cstarlattice(A)$ satisfies $g\preceq f_i$ for all $i\in I$, then $g\preceq f$.
  \end{enumerate}
Again it is straightforward to verify that, if such an element $f$ exists, then it is unique. Another term for the meet is \textit{greatest lower bound}. The meet will be concretely described in terms of boundary ideals in the next subsection. 

A partially ordered set in which every finite subset has a join and a meet is called a \textit{lattice}. In any lattice, the join and meet are traditionally denoted by $\vee$ and $\wedge$, respectively. If every (possibly infinite) subset in a partially ordered set has a join and a meet, then it is called \textit{complete lattice}. %It is well known that $\cstarlattice(A)$ forms a complete lattice with respect to $\preceq$, for a proof see \cite{Hamidi} or \cite{Thompson}.
\begin{proposition}[\cite{Hamidi} and \cite{Thompson}]
    Let $A$ be an operator algebra. Then $\cstarlattice(A)$ with the ordering $\preceq$ forms a complete lattice.
\end{proposition}

Every complete lattice has a unique minimum and maximum element, the meet and join of the whole lattice. With a slight abuse of notation, we will call the minimum element of $\cstarlattice(A)$, and every $C^*$-cover in the corresponding equivalence class, the \textit{$C^*$-envelope} of $A$, denoted by 
  \[
    [C^*_e(A),i_e]\qquad \textup{respectively}\qquad (C^*_e(A),i_e). 
  \]
Which of these two notations is meant will be clear from the context. Nevertheless, it is important to note that, in terms of $C^*$-covers, the $C^*$-envelope is unique up to isomorphism and enjoys the universal property that given any C*-cover $(\A,j)$ of $A$, there is a unique $\ast$-homomorphism $\pi:\A\to C^*_e(A)$ such that $\pi j = i_e$. 
The existence of the $C^*$-envelope is non-trivial and was conjectured by Arveson in \cite{Arv1} and first proved by Hamana in \cite{Hamana}.

Likewise, we will call the maximum element of $\cstarlattice(A)$ the \textit{maximal $C^*$-cover} of $A$, denoted by 
  \[
    [C^*_{max}(A),i_{max}]\qquad \textup{respectively}\qquad (C^*_{max}(A),i_{max}). 
  \]
A proof of its existence can be found in \cite{Blecher}. The following characterization of the maximal $C^*$-cover is well known and will be referred to as \textit{the universal property} of the maximal $C^*$-cover.
\begin{lemma}\label{lem:uni property}
    Let $(\mathcal{A},\pi)$ be a $C^*$-cover of an operator algebra $A$. The following are equivalent:
    \begin{enumerate}
        \item $(\mathcal{A},\pi)$ is maximal with respect to $\preceq$.
        \item For every u.c.c. homomorphism $\rho:A\to\mathcal{B}(H)$ there exists a $*$-homomorphism $\phi:\mathcal{A}\to\mathcal{B}(H)$ such that 
          \[
            \rho=\phi\pi.
          \]
    \end{enumerate}
\end{lemma}
%We will later also give a characterization of the $C^*$-envelope.

Finally, we end this subsection with a brief description of what is known about C*-covers. \cite{KatRamMem} and \cite{Hamidi} show that completely isometric automorphisms of $A$ need not lift to all C*-covers. \cite{HumRam} presents four related but distinct equivalence relations capturing ``sameness" for lattices of C*-covers, all of which are weaker than completely isometrically isomorphism. \cite{HKR} shows that the semi-Dirichlet C*-covers of a semi-Dirichlet operator algebra form a complete sublattice. \cite{HumRamThom} shows that the RFD (residually finite-dimensional) C*-covers of an RFD operator algebra often fail to form a complete sublattice.

\subsection{Boundary ideals}
Let $(\mathcal{A},i)$ be a $C^*$-cover of $A$. A closed two-sided ideal $I\subset\mathcal{A}$ is called a \textit{boundary ideal} if the quotient map
  \[
    \mathcal{A}\to\mathcal{A}/I
  \]
is completely isometric on $i(A)$. A boundary ideal that contains every other boundary ideal is called the \textit{Shilov boundary ideal}, or for short the \textit{Shilov ideal}. Note that boundary ideals are always considered with respect to a specific $C^*$-cover, and that the Shilov ideal is unique. 
The existence of the Shilov ideal is equivalent to the existence of the $C^*$-envelope, and hence the Shilov ideal always exists.

The set of boundary ideals in a fixed $C^*$-cover, together with the canonical ordering by inclusion, forms a complete lattice, where the meet of a family $(I_i)_i$ is given by $\bigcap_i I_i$ and the join by the (closed two-sided) ideal $\overline{\sum_i I_i}$ generated by $\bigcup_i I_i$. This complete lattice has a natural connection to $\cstarlattice(A)$, given by the following proposition. For a proof see \cite[Theorem 3.3]{Thompson}.

\begin{proposition}\label{Prop:Equi Lattice Ideals}
    Let $A$ be an operator algebra and $(C^*_{max}(A),i_{max})$ a maximal $C^*$-cover. Then the maps
      \begin{align*}
        \cstarlattice(A)&\to \{\textup{boundary ideals in $C^*_{max}(A)$}\},\\ [\mathcal{B},j]&\longmapsto \ker(\tilde j),
      \end{align*}
    where $\tilde j:C^*_{max}(A)\to\mathcal{B}$ is the unique unital $*$-homomorphism such that
      \[
        \tilde j \circ i_{max}=j,
      \]
      and 
      \[
        \{\textup{boundary ideals in $C^*_{max}(A)$}\}\to\cstarlattice(A),\quad I\longmapsto (C^*_{max}(A)/I,q_I),
      \]
      where $q_I$ is the quotient map by $I$, are order-reversing and mutually inverse.
\end{proposition}

We also have the following characterization of the $C^*$-envelope.

\begin{lemma}\label{lem:empty Shilov}
    Let $A$ be an operator algebra and $(\mathcal{A},i)$ a $C^*$-cover of $A$. Then the following are equivalent:
      \begin{enumerate}
          \item $(\mathcal{A},i)$ is the $C^*$-envelope of $A$.
          \item $\A$ contains no nonzero boundary ideals for $i(A)$.
          \item The Shilov ideal with respect to $(\mathcal{A},i)$ is $\{0\}$.
      \end{enumerate}
\end{lemma}

\subsection{Maximal u.c.c. maps}

Let $H$ and $K$ be Hilbert spaces with $H\subset K$, and let $\pi:A\to\mathcal{B}(H)$ and $\phi:A\to\mathcal{B}(K)$ be u.c.c. maps. We say that $\phi$ \textit{dilates} $\pi$, or that $\phi$ is a \textit{dilation} of $\pi$, if 
  \[
    \pi=P_H\phi|_H.
  \]
We say that $\pi$ is \textit{maximal} if $H$ is a reducing subspace for every dilation of $\pi$. For a proof see \cite[Theorem 1.2]{DriMcc}.

\begin{proposition}\label{prop:existence max dilation}
    Let $A$ be an operator algebra and $\pi$ a u.c.c. map on $A$. Then there exists a dilation $\phi$ of $\pi$ such that $\phi$ is maximal.
\end{proposition}

Together with part $(i)$ of the next lemma, this proposition gives rise to a proof of the existence of the $C^*$-envelope.

\begin{lemma}\label{lem: maximal u.c.c. and envelope}
    Let $A$ be an operator algebra and let $\pi$ be a maximal u.c.i. homomorphism on $A$. Then 
    \begin{enumerate}
        \item $(C^*(\pi(A)),\pi)$ is a $C^*$-envelope of $A$.
        \item If $(C^*_e(A),i_e)$ is the $C^*$-envelope of $A$, then there exists a unique *-homomorphism $\tilde\pi$ on $C^*_e(A)$ such that
          \[
            \tilde\pi\circ i_e=\pi.
          \]
    \end{enumerate}
\end{lemma}

So if we take an u.c.i. homomorphism $\rho$ and dilate it to a maximal u.c.c. homomorphism $\pi$, then it is straightforward to check that $\pi$ is also u.c.i., and hence by (i) in the previous lemma, $(C^*(\pi(A)),\pi)$ is a $C^*$-envelope of $A$. We will use this construction later, choosing $\rho$ as the embedding into the maximal $C^*$-cover.

\begin{lemma}\label{lem:intersection shilov ideal}
    Let $A$ be an operator algebra, $(C^*_{max}(A),i_{max})$ the maximal $C^*$-cover and $I$ the Shilov ideal. For a u.c.c. homomorphism $\pi$ on $A$ denote by $\tilde{\pi}$ the unique unital $*$-homomorphism on $C^*_{max}(A)$ with 
      \[
        \tilde{\pi}i_{max}=\pi,
      \]
      then 
      \[
        I=\bigcap\{\ker(\tilde{\pi})\ |\ \pi:A\to\mathcal{B}(H)\ \textup{u.c.c. and maximal}\}.
      \]
\end{lemma}

Combining this with Lemma \ref{lem:empty Shilov}, we obtain that if $(C^*_{max}(A),i_{max})$ is not the $C^*$-envelope, then $I \neq \{0\}$. This observation will play a crucial role in the proof of the Dichotomy Theorem.

\begin{lemma}\label{lem:maximal_SOT_limit}
Let $A$ be an operator algebra in a $C^*$-algebra $\mathcal{A}=C^*(A)$ and let $\pi, \pi_n:\mathcal{A}\to\mathcal{B}(H)$ be unital *-homomorphisms, such that $\pi_n\to\pi$ pointwise *-SOT. If each $\pi_n|_A$ is maximal, then there exists a $*$-homomorphism $\psi$ defined on the C*-envelope $(\Ce(A),i_e)$ such that
  \[
    \psi\circ i_e=\pi.
  \]
\end{lemma}

\begin{proof}
 Since every homomorphism $\pi_n\vert_A$ is maximal, by Lemma \ref{lem: maximal u.c.c. and envelope}, there are *-homomorphisms $\psi_n$ defined on $\Ce(A)$ with 
   \[
     \psi_n\circ i_e=\pi_n\vert_A. 
   \]
Let $q:\A\to C^*_e(A)$ be the quotient map, which satisfies $q\vert_A = i_e$. Then $\psi_n \circ q = \pi_n$ because these are $\ast$-homomorphisms equal on $A$ and $\A=C^*(A)$.

It suffices to show that $(\psi_n)_n$ is a pointwise *-SOT Cauchy-sequence. Let $a\in \Ce(A)$. Since $q$ is surjective, there is $\tilde a\in\mathcal{A}$ such that $q(\tilde a)=a$. Then,
  \begin{equation*}
  \|(\psi_n-\psi_m)(a)x\|=\|(\pi_n-\pi_m)(\tilde a)x\|
  \end{equation*}
for all $x\in H$. Hence the sequence $(\psi_n)$ has a *-SOT limit $\psi$, and standard arguments show that $\psi$ is again a *-homomorphism. For $a\in A$ we have
  \[
    \psi(i_e(a))=\lim_n\psi_n(i_e(a))=\lim_n\pi_n(a)=\pi(a),
  \]
so $\psi\circ i_e=\pi|_A$, as required.
\end{proof}

\subsection{Direct limits}
A sequence of operator algebras $(A_n)_{n\in\mathbb{N}}$ together with u.c.i. homomorphisms
  \[
    i_n:A_n\to A_{n+1}
  \]
will be called a \textit{direct system} of operator algebras. There are several ways to construct their \textit{direct limit}. One way, described in \cite[Example 2.3.8]{Blecher}, is by considering the $A_n$ as Banach algebras, building their direct limit $A$ in the Banach algebra sense, turning $M_m(A)$ into Banach algebras using the operator algebra structure of the $A_n$, and then using the Blecher-Ruan-Sinclair Theorem to verify that this is again an (abstract) operator algebra. This approach does not give enough information for our purposes since we are interested in an explicit connection between the $C^*$-covers of the $A_n$ and the $C^*$-covers of the direct limit, which in the above construction do not come into effect. So we will present a more concrete way to obtain the direct limit with the help of $C^*$-covers, similar to \cite[Proposition 16] {KirWas}, and build a connection between the $C^*$-covers of the sequence and the $C^*$-covers of the direct limit. Because we need no further generality, we present this construction only for direct systems indexed by the natural numbers, but this can all be extended without obstruction to direct systems indexed by any directed set.

For each $n$, let $C^*(A_n)$ be an arbitrary $C^*$-cover of $A_n$, where to simplify notation greatly we omit the embedding and view $A_n$ as a subset of $C^*(A_n)$. Let $\mathcal{A}$ be the $\ell^\infty$-direct product of the $C^*(A_n)$, and consider the ideal 
  \[
    I=\{(x_n)\in \mathcal{A}\ |\ \lim_{n\to\infty}\|x_n\|=0\}.
  \]
Let $B\subset\mathcal{A}$ be the algebra of sequences $(x_n)$ such that 
  \begin{enumerate}
    \item $x_n\in A_n\subseteq C^*(A_n)$ for all $n$,
    \item $x_{n+1}=x_n$ eventually.
  \end{enumerate}
Set
  \[
    A=\overline{(B+I)/I}\subset \mathcal{A}/I.
  \]
For $n\ge0$ define $\lambda_n:A_n\to A$ by $x\mapsto (x_j)+I$, where
  \begin{equation*}
      x_j=\begin{cases} 1 & (j<n),\\ x & (j=n), \\ (i_{j-1}i_{j-2}\dots i_n)(x) & (j>n). \end{cases}
  \end{equation*}
It is straightforward to verify that each $\lambda_n$ is a u.c.i. homomorphism,  $\lambda_n=\lambda_{n+1}i_n$ and $A=\overline{\bigcup_{n=1}^\infty\lambda_n(A_n)}$. We call the operator algebra $A$ the \textit{direct limit} of $(A_n,i_n)_n$ and the $C^*$-algebra $C^*(A)\subset\mathcal{A}/I$ together with the identity embedding the \textit{induced} $C^*$-cover of $A$.

Now every sequence of $C^*$-covers $(\mathcal{A}_n,\phi_n)_n$ of $(A_n)_n$ induces a direct system via
  \[
    \phi_0(A_0)\xrightarrow{\tilde{i}_0} \phi_1(A_1)\xrightarrow{\tilde{i}_1}
    \phi_2(A_2)\xrightarrow{}\dots
  \]
where 
  \begin{align*}
\tilde{i}_n:\phi_n(A_n)&\to\phi_{n+1}(A_{n+1}), \\ a&\mapsto(\phi_{n+1} i_n \phi_n^{-1})(a).
  \end{align*}
It should not be surprising, and will be shown in the following proposition, that the direct limit of this direct system coincides with $A$ in the sense of the existence of a surjective u.c.i. homomorphism between these two direct limits. Hence the induced $C^*$-cover of $(\phi_n(A_n),\tilde{i}_n)_n$ is also a $C^*$-cover of $A$, however, dependent on the surjective u.c.i. homomorphism. To circumvent this problem, we will require the surjective u.c.i. homomorphism to have a certain intertwining property that makes it unique, as shown in the following proposition.

\begin{proposition}\label{Prop:Direct systems}
Given a direct system $(A_n,i_n)_n$ of operator algebras and $C^*$-covers $(\mathcal{A}_n,\phi_n)$. Let $A$ be the direct limit of $(A_n,i_n)_n$ and $B$ the direct limit of $(\phi_n(A_n),\tilde{i}_n)_n$. Denote by 
  \[
    \lambda_n:A_n\to A\quad \textup{and}\quad \tilde{\lambda}_n:\phi_n(A_n)\to B
  \]
the u.c.i. homomorphisms from the construction of the direct limit. Then there is a unique surjective u.c.i. homomorphism $\pi:A\to B$ such that $\pi\lambda_n=\tilde{\lambda}_n\phi_n$.

Additionally, if an operator algebra $C$ fulfills
  \begin{enumerate}
      \item For every $n$ there exist u.c.i. homomorphisms $\rho_n:A_n\to C$ such that $\rho_n=\rho_{n+1}i_n$,
      \item $\overline{\bigcup_{n=1}^\infty\rho_n(A_n)}=C$,
  \end{enumerate}
then there exists a unique surjective u.c.i. homomorphism $\gamma:A\to C$ such that $\gamma\lambda_n=\rho_n$.
\end{proposition}

In particular, this shows that the resulting direct limit operator algebra $A$ does not depend on the particular C*-covers used for each $A_n$, up to completely isometric isomorphism.

\begin{proof}
With the above notations, we have for every $n$ that the map
  \[
    j_n:\lambda_n(A_n)\to\tilde{\lambda}_n(\phi_n(A_n)), \ \lambda_n(a)\mapsto\tilde{\lambda}_n(\phi_n(a))
  \]
is a surjective u.c.i. homomorphism and 
  \[
    j_{n+1}\lambda_n=j_{n+1}\lambda_{n+1}i_n=\tilde{\lambda}_{n+1}\phi_{n+1}i_n=\tilde{\lambda}_{n+1}\tilde{i}_n\phi_n=\tilde{\lambda}_n\phi_n=j_n\lambda_n
  \]
Thus the maps $j_n$ induce a well-defined surjective u.c.i. homomorphism
  \[
  \bigcup_{n=1}^\infty\lambda_n(A_n)\to\bigcup_{n=1}^\infty\tilde{\lambda}_n(\phi_n(A_n))
  \]
that extends to a surjective u.c.i. homomorphism $\pi:A\to B$ with $\pi\lambda_n=\tilde{\lambda}_n\phi_n$. The uniqueness of $\pi$ follows from the fact that $\bigcup_{n=1}^\infty\lambda_n(A_n)$ is dense in $A$, proving the first part of the proposition.

The additional claim follows analogously to the above proof by replacing $\tilde{\lambda}_n\phi_n$ with the maps from (i).
\end{proof}

In the setting of the above proposition, we will call the $C^*$-cover $(C^*(B),\pi)$ the \textit{induced $C^*$-cover of $A$} with respect to $C^*$-covers $(C^*(\phi_n(A_n)),\phi_n)_n$. Note that by the previous proposition the u.c.i. homomorphism $\pi$ fulfills
  \[
    \pi\lambda_n=\tilde{\lambda}_n\phi_n.
  \]

\begin{theorem}\label{theorem:commuting diagram}
Let $(A_n,i_n)_n$ be a direct system with direct limit $A$, let $(\mathcal{A}_n,\varphi_n)_n$ be $C^*$-covers of $(A_n)_n$, and assume that there are $*$-isomorphisms $\pi_n:C^*(A_n)\to \mathcal{A}_{n+1}$ such that the diagram
\begin{equation}\label{diagram:cover_commuting}
\begin{tikzcd}
    & \varphi_{n+1}(A_{n+1}) \\
    A_n \arrow[r, "i_n" ] \arrow[ur, "\pi_n" ] & A_{n+1} \arrow[u,"\varphi_{n+1}",right]
\end{tikzcd}
\end{equation}
commutes. Then the two induced $C^*$-covers of $(A_n,i_n)_n$ and $(\phi_n(A_n),\tilde{i}_n)_n$ of $A$ are equivalent, where $\tilde{i}_n:\varphi_n(A_n)\to\varphi_{n+1}(A_{n+1})$ are the induced u.c.i. homomorphisms given by $\phi_{n+1}i_n\phi_n^{-1}$.
\end{theorem}

\begin{proof}
With the above notations, let $\mathcal{B}_1$ be the $\ell^\infty$-sum of $(C^*(A_n))_n$ and $\mathcal{B}_2$ the $\ell^\infty$-sum of $(\mathcal{A}_n)_n$. It is clear that 
  \begin{align*}
  \Phi:\mathcal{B}_1&\to\mathcal{B}_2, \\ (x_n)&\mapsto (0,\pi_0(x_0),\pi_1(x_1),\dots )
  \end{align*}
is an injective $*$-homomorphism. Set, for $i=1,2$,
  \[
    I_i=\{(x_n)\in\mathcal{B}_i\ |\ \lim_{n\to\infty}\|x_n\|=0\}.
  \]
It is straightforward to verify that
  \begin{align*}
  \pi:\mathcal{B}_1/I_1&\to\mathcal{B}_2/I_2, \\ (x_n)+I_1&\mapsto (x_n)+I_2
  \end{align*}
is well-defined and a $*$-isomorphism.

Recall that $C^*(A)\subset\mathcal{B}_1/I_1$ together with the identity embedding is the induced $C^*$-cover of $(A_n,i_n)_n$, and denote by $(\tilde{\mathcal{A}},j)$ the induced $C^*$-cover of $A$ with respect to $(\mathcal{A}_n,\phi_n)_n$. We have by construction of the induced $C^*$-cover that $\tilde{\mathcal{A}}\subseteq\mathcal{B}_2/I_2$. To obtain that these two $C^*$-covers are equivalent, it suffices to show that
  \[
    \pi|_A=j.
  \]
Let $\lambda_n$ and $\tilde{\lambda}_n$ be the u.c.i. homomorphisms to each direct limit operator algebra resulting from the construction of the direct limit with respect to $(A_n,i_n)$, and $(\phi_n(A_n),\tilde{i}_n)$. Using that $\bigcup_{n=1}^\infty\lambda_n(A_n)$ is dense in $A$ and the definition of $\lambda_n$ and $\tilde{\lambda}_n$, as well as $j\lambda_n=\tilde{\lambda}_n\phi_n$, it is enough to show that for every $m>n$ and $x\in A_n$
  \begin{equation}\label{Eq:Direct limit}
    (\pi_{m-1}i_{m-2}i_{m-3}\dots i_n)(x)=(\tilde{i}_{m-1}\tilde{i}_{m-2}\dots \tilde{i}_n)(\varphi_n(x)).
  \end{equation}
First we use that $\tilde{i}_n\phi_n=\phi_{n+1}i_n$ to obtain
  \[
    (\tilde{i}_{m-1}\tilde{i}_{m-2}\dots \tilde{i}_n\varphi_n)(x)=(\varphi_mi_{m-1}i_{m-2}\dots i_n)(x).
  \]
Now the assumption that the diagram \ref{diagram:cover_commuting} commutes yields that $\varphi_mi_{m-1}=\pi_{m-1}$, showing that Equation \eqref{Eq:Direct limit} holds, and therefore the theorem is proven. 
\end{proof}

Proposition \ref{Prop:Direct systems} allows us to view the operator algebras $A_n$ as subalgebras of $A$, ordered by inclusion. This allows us to omit the embeddings $i_n:A_n\to A_{n+1}$ and $\lambda_n:A_n\to A$, which we will do for the rest of this section.

So far, we have shown that choosing specific C*-covers of each operator algebra $A_n$ naturally leads to a C*-cover of the resulting direct limit operator algebra $A$. The next result states that every C*-cover of $A$ arises in this fashion.

\begin{lemma}\label{lem:C*-cover direct system}
Let $(A_n,i_n)$ be a direct system of operator algebras and $(\mathcal{A},j)$ be a $C^*$-cover of the direct limit $A$. Then the induced $C^*$-cover of $A$ with respect to the $C^*$-covers $(C^*(j(A_n)),j|_{C^*(A_n)})$ is equivalent to $(\mathcal{A},j)$.
\end{lemma}

\begin{proof}
Define $\mathcal{A}_n=C^*(j(A_n)), j_n=j|_{C^*(A_n)}$. Denote the induced $C^*$-cover of $A$ with respect to the $C^*$-covers $(\mathcal{A}_n,j_n)$ by $(\tilde{\mathcal{A}},\pi)$, and note that by the intertwining property of $\pi$ the u.c.i. homomorphisms from the construction of the direct limit of $(j(A_n),ji_{n}j^{-1})_n$ are given by $\pi\circ j_n^{-1}:j(A_n)\to\tilde{\mathcal{A}}$. We have to show that there exists a $*$-isomorphism $\Phi:\tilde{\mathcal{A}}\to\mathcal{A}$ such that 
  \[
    \Phi\circ \pi=j.
  \]
The sequence of $C^*$-algebras $(\mathcal{A}_n)_n$ is ordered by inclusion and therefore has a direct limit in the sense of $C^*$-algebras for which it is straightforward to verify that it coincides with $\tilde{\mathcal{A}}$ and the embedding of $\mathcal{A}_n\to\tilde{\mathcal{A}}$ extends $\pi\circ j_n^{-1}$. Furthermore, by \cite[Theorem 6.1.2 and Remark 6.1.3]{Murphy}, there exists a $*$-isomorphism $\Phi:\mathcal{A}\to\tilde{\mathcal{A}}$ that intertwines the embedding of $\mathcal{A}_n\subset\mathcal{A}$ and $\mathcal{A}_n\to\tilde{\mathcal{A}}$. Hence
  \[
    \Phi(j(a))=\Phi(j_n(a))=(\pi j^{-1}_n)(j(a))=\pi(a)
  \]
for all $a\in A_n$, and since $\bigcup_{n=1}^\infty A_n$ is dense in $A$, we obtain equivalence of the $C^*$-covers.
\end{proof}

\begin{theorem}\label{the:min/max C*cover direct system}
The induced $C^*$-cover of $A$ obtained by using the maximal $C^*$-covers of each $A_n$ yield a maximal $C^*$-cover of $A$. Analogously, the $C^*$-envelopes of every $A_n$ yield the $C^*$-envelope of $A$.
\end{theorem}

\begin{proof}
It is straightforward to check that if $\mathcal{A}_n$ and $\mathcal{B}_n$ are $C^*$-covers of $A_n$ such that $\mathcal{A}_n\preceq \mathcal{B}_n$ then $\mathcal{A}\preceq \mathcal{B}$, where $\mathcal{A}$ is the induced $C^*$-cover corresponding to $(\mathcal{A}_n)_n$ and $\mathcal{B}$ the one to $(\mathcal{B}_n)_n$. Since by Lemma \ref{lem:C*-cover direct system}, every $C^*$-cover of $A$ is equivalent to a induced $C^*$-cover, the result follows.
\end{proof}

As a nice consequence, the expected universal property of any direct limit now follows immediately by extending all u.c.c. homomorphisms involved to the respective maximal C*-algebras.

\begin{lemma}\label{lemma: direct system lim homomorphisms}
    Let $(A_n,i_n)$ be a direct system, $A$ its direct limit and $\pi_n$ u.c.c. homomorphisms on $A_n$ such that 
      \[
        \pi_{n+1}\circ i_n=\pi_n.
      \]
    Then there exists a u.c.c. homomorphism $\pi$ on $A$ such that
      \[
        \pi\circ \lambda_n=\pi_n,
      \]
      where $\lambda_n:A_n\to A$ are the u.c.i. homomorphisms from the construction of $A$.
\end{lemma}

\section{Non-trivial one point lattice example}\label{Section:non-trivial one point lattice example}

 In this section, we will give a general construction of a non-selfadjoint operator algebra with only one $C^*$-cover, i.e. the maximal and minimal C*-covers coincide, thus answering Question 3.1 in \cite{HumRam}. It shall be noted that our construction is inspired by the one in \cite[Proposition 16]{KirWas}, which Kirchberg and Wassermann used to construct operator systems with only one C*-cover. Note that our results are novel and cannot be obtained as a corollary of the Kirchberg-Wassermann construction, because their construction takes place in the operator system category in which multiplicative structure is not preserved.

Before proceeding, we isolate a result that will be used to obtain (closed) operator algebras repeatedly throughout.

\begin{lemma}\label{lem:isometric_quotient_closed}
Let $X$ be a Banach space, and let $Y,Z\subseteq X$ be closed subspaces. If the quotient map $X\to X/Y$ is isometric on $Z$, then $Z+Y$ is closed.
\end{lemma}

\begin{proof}
Let $q:X\to X/Y$ be the quotient map. Because $q$ is isometric on $Z$, the image $q(Z)$ is closed. So, $Z+Y=q^{-1}(q(Z))$ is the continuous preimage of a closed set, which is closed.
\end{proof}

Our construction will require the following approach to enlarging an operator algebra. Given an operator algebra concretely embedded in some C*-cover, we can extend it by adding the corresponding Shilov ideal.

\begin{proposition}\label{prop:shilov_extension}
Let $A\subseteq C^\ast(A)$ be a concrete operator algebra with Shilov ideal $I\triangleleft C^\ast(A)$. The set $A+I$ is a closed subalgebra containing $I$ as a self-adjoint ideal, and the sum $A+I$ is direct in the sense that $A\cap I=0$. The projection map $p:A+I\to A$ given by $p\vert_A=\id_A$ and $p\vert_I=0$ is a completely contractive homomorphism.

The Shilov ideal for $A+I$ in $C^\ast(A)$ is $\{0\}$, and so $C^\ast(A)$ is the C*-envelope of $A+I$. Also, if $A$ is non-self-adjoint, then so is $A+I$.
\end{proposition}

\begin{proof}
Because $I$ is an ideal, it follows that $A+I$ is a subalgebra containing $I$ as an ideal. Because the quotient map $q:C^\ast(A)\to C^\ast(A)/I$ is completely isometric on $A$, and $q$ annihilates $I$, any element $x\in A\cap I$ satisfies $\|x\|=\|q(x)\|=\|0\|=0$, so $A\cap I=0$. Moreover, because $q$ is isometric on $A$, the subalgebra $A+I$ is closed by Lemma \ref{lem:isometric_quotient_closed}. Restricting $q$ to $A+I$ gives a completely contractive homomorphism $p:A+I\to q(A)=(A+I)/I\cong A$ satisfying $p\vert_A=\id$ and $p\vert_I=0$.

Now, let $J\triangleleft C^\ast(A)$ be a boundary ideal for $A+I$. Since $A\subseteq A+I$, the ideal $J$ is also a boundary ideal for $A$, and so $J\subseteq I$. But, the quotient map $C^\ast(A)\to C^\ast(A)/J$ must be completely isometric on $I\subseteq A+I$, and so $I\cap J=\{0\}$. Therefore $C^\ast(A)$ is the C*-envelope for $A+I$.

Finally, suppose $A+I$ is self-adjoint, then it would be a C*-algebra and hence $C^\ast(A)=A+I$. Therefore the quotient map $A+I=C^\ast(A)\to C^\ast(A)/I \cong C^\ast_e(A)$ is a $\ast$-homomorphism whose range is the copy of $A$ in $C_e^\ast(A)$. Therefore $A$ itself is self-adjoint.
\end{proof}

Given an abstract operator algebra, which is represented uniquely only up to completely isometric isomorphism, this amounts to fixing a given C*-cover $(C,\iota)$ of $A$, and passing to the subalgebra $\iota(A)+I$. Frequently, in proofs, it is usually simpler to consider $\iota$ to be the identity map, which simplifies notation greatly without losing generality.

\begin{example}
Consider the operator algebra $T_2\subseteq M_2$, and the completely contractive homomorphism $\pi:T_2\to \bbC$ given by
\[
\pi\begin{pmatrix}
    a & b \\
     & c
\end{pmatrix} =
a.
\]
Then, with $\id:T_2\to M_2$ the inclusion map, consider the operator algebra
\[
A = 
(\pi\oplus \id)(T_2) =
\left\{\begin{pmatrix}
    a & & \\
     & a & b \\
    & & c 
    \end{pmatrix} \;\middle|\;
    a,b,c\in \bbC\right\}\subseteq T_2.
\]
The operator algebra $A$ is just a completely isometric copy of $T_2$, but embedded in a C*-algebra $\bbC\oplus M_2$ which is not its envelope. Here, the Shilov ideal is $I=\bbC \oplus 0$, and the extension of $A$ by $I$ is
\[
A+I = \bbC \oplus T_2,
\]
and indeed $C^\ast(A+I)=\bbC\oplus M_2$ is the C*-envelope. By \cite[Corollary 2.8]{HumRam}, $T_2$ and the direct sum $\bbC\oplus T_2$ have isomorphic C*-cover lattices. The main result of Section \ref{Section:Extension of an operator algebra by a Shilov ideal}, Theorem \ref{thm:shilov_extension_lattice_isomorphism}, is that this always occurs: $A$ and $A+I$ are lattice isomorphic.

Note that even though $A\cap I = \{0\}$ in the operator algebra $A+I$, the operator algebra is not completely isometrically isomorphic to the operator algebra direct sum $A\oplus I$ via the map $(a,x)\to a+x$ for $a\in A$ and $x\in I$.
\end{example}

\begin{example}
By \cite[Example 2.4]{Blech}, the maximal C*-cover of $T_2$ is $(C^\ast_{\max}(T_2),i_{\max})$ can be described explicitly as
\[
C^\ast_{\max}(T_2) =
\{F\in C([0,1],M_2) \;\mid \; F(0) \text{ is diagonal}\},
\]
with embedding
\[
i_{\max}\left(\begin{pmatrix}
    a & b \\
    0 & c
\end{pmatrix}\right)(t) =
\begin{pmatrix}
    a & b\sqrt{t} \\
    0 & c
\end{pmatrix}.
\]
Let $A:=i_{\max}(T_2)$. The unique C*-cover morphism down to the C*-envelope $C^\ast_e(T_2)\cong M_2$ is just the evaluation map $F\mapsto F(1)$. So, the Shilov ideal for $A$ is
\[
I = 
\{F\in C([0,1],M_2)\;\mid \; F(0) \text{ is diagonal and }F(1)=0\}.
\]
The extension of $A$ is therefore
\[
A+I =
\{F\in C([0,1],M_2) \;\mid\; F(0)\text{ is diagonal and } F(1) \text{ is upper triangular}\}.
\]
There is no closed subset of $[0,1]$ for which restriction is completely isometric on $A+I$, and so the Shilov ideal for $A+I$ in $C^\ast_{\max}(T_2)$ is trivial.
\end{example}

Next we will give a full characterization of every u.c.c. homomorphism on $A+I$ with respect to a u.c.c. homomorphism on $A$, which we will use several times throughout the rest of this section. The result should be compared to the decomposition result preceding Theorem 1.3.4 in \cite{Arv2}.

\begin{theorem}\label{ucc Decom}
Let $A$ be a non-selfadjoint operator algebra in $C^*(A)$ and $I$ the Shilov ideal of $A$. Then every u.c.c. homomorphism $\pi$ on $A+I$ has a decomposition
      \begin{equation*}
          \pi=\pi_1\oplus\pi_2
      \end{equation*}
into u.c.c. homomorphisms such that $\pi_1(I)=0$ and $\pi_2$ extends to a unital $*$-homomorphism on $C^*(A)$. Moreover, for every u.c.c. homomorphism $\pi$ on $A$, there is a u.c.c homomorphism $\tilde\pi$ on $A+I$ such that $\tilde\pi(I)=0$ and $\tilde\pi(a)=\pi(a)$ for all $a\in A$.
\end{theorem}

\begin{proof}
    If $\pi:A+I\to \mathcal{B}(H)$ is a u.c.c. homomorphism on $A+I$, then $K=\overline{\pi(I)(H)}$ is a reducing subspace. Indeed, we have, for all $a\in A+I, i\in I, x\in K, y\in K^\perp$ that
      \begin{equation*}
          \langle\pi(a)\pi(i)x,y\rangle=\langle\pi(ai)x,y\rangle=0
      \end{equation*}
    since $\pi(ai)x\in K$, and
       \begin{equation*}
          \langle\pi(i)x,\pi(a)y\rangle=\langle\pi(a)^*\pi(i)x,y\rangle=\langle(\pi(i^*a)^*x,y\rangle=\langle \pi(a^*i)x,y\rangle=0,
      \end{equation*}
where we used that $\pi$ restricted to $I$ is a $*$-homomorphism. 

Let 
      \begin{equation*}
          \pi=\pi_1\oplus\pi_2
      \end{equation*}
be the decomposition obtained by compressing to $K$ respectively $K^\perp$. It is straightforward to check that $\pi_2$ vanishes on $I$. On the other hand, since $\pi_1$ is non-degenerate by construction, it follows from \cite[Lemma 1.9.14]{Dav} that $\pi_1$ extends to a unital $*$-homomorphism on $C^*(A)$.

The additional note follows from the observation that the operator algebra $(A+I)/I$ can be identified with $A$ via an unital completely isometric homomorphism. Thus, if $\pi$ is a u.c.c. homomorphism on $A$, one can choose $\tilde \pi$ as the composition of $\pi$ with the quotient map.
\end{proof}

In Section \ref{Section:Extension of an operator algebra by a Shilov ideal}, we completely identify the structure of the C*-cover lattice of the extension $A+I$ by explicitly constructing an order isomorphism between it and the C*-cover lattice for $A$, however, we do not need this isomorphism in the proof of the following Kirchberg-Wassermann \cite[Proposition 16]{KirWas} type result.

\begin{theorem}\label{KW construction}
Every non-selfadjoint operator algebra has a unital completely isometric embedding into a non-selfadjoint operator algebra with only one $C^*$-cover.
\end{theorem}

\begin{proof} Let $A_0$ be a unital, non-selfadjoint operator algebra. After identifying $A_0$ with its image in $C^*_e(A_0)$, we may assume that $C^*(A_0)$ is already the $C^*$-envelope. Fix a maximal $C^*$-cover $(C^*_{max}(A_0),i_0)$ and let $I_0$ denote the Shilov ideal of $i_0(A_0)$ in $C^*_{max}(A_0)$. 

For $n\ge0$,  proceed inductively. Given a maximal $C^*$-cover $(C^*_{max}(A_n),i_n)$ of $A_n$ and the Shilov ideal $I_n\subset C^*_{max}(A_n)$ of $i_n(A_n)$, set
  \[
    A_{n+1}=i_n(A_n)+I_n\subset C^*_{max}(A_n),
  \]
This yields a direct system of operator algebras
  \[
    i_n:A_n\to A_{n+1},
  \]
and the inductive system
  \begin{equation*}
      A_0\xrightarrow{i_0} A_1\xrightarrow{i_1} A_2\xrightarrow{i_2}\dots
  \end{equation*}
Let $A$ be the direct limit operator algebra of the sequence $(A_n)_n$ and let $\lambda_n:A_n\to A$ be the u.c.i. homomorphisms such that
  \begin{enumerate}
    \item $\lambda_n=\lambda_{n+1}i_n$,
    \item $A=\overline{\bigcup_{n=0}^\infty\lambda_n(A_n)}$.
  \end{enumerate}
By Lemma \ref{prop:shilov_extension}, $(C^*_{max}(A_n),\text{id})$ is an $C^*$-envelope of $A_{n+1}$. Applying Theorem \ref{the:min/max C*cover direct system} to the system $(A_n,i_n)_n$, we obtain that the induced $C^*$-cover is an $C^*$-envelope of $A$.

By the same theorem, we obtain a maximal $C^*$-cover of $A$ by the induced $C^*$-cover of $(C^*(i_n(A_n)),i_n)$. In this case, the u.c.i. homomorphism from $i_n(A_n)\to i_{n+1}(A_{n+1})$ is given by $i_{n+1}i_ni_n^{-1}=i_{n+1}$. Since the diagram
\[
\begin{tikzcd}
    & A_{n+1} \\
    i_n(A_n) \arrow[r, "i_{n+1}", swap ] \arrow[ur, "id" ] & i_{n+1}(A_{n+1}) \arrow[u,"i_{n+1}^{-1}",right,swap]
\end{tikzcd}
\]
commutes, we can apply Theorem \ref{theorem:commuting diagram} and obtain that the maximal $C^*$-cover and the $C^*$-envelope of $A$ are equivalent.

It remains to show that $A$ is non-selfadjoint. We will show that $A$ has at least one non-maximal representation. For any fixed $n\in\mathbb{N}$, assume that we have u.c.c. homomorphisms $\pi_n$ and $\Phi_n$ on $A_n$ such that $\Phi_n$ is a dilation of $\pi_n$. By Theorem \ref{ucc Decom}
  \begin{align*}
    \pi_{n+1}:i_n(A_n)+I_n&\to\mathcal{B}(H), \\i_n(x)+j&\mapsto \pi_n(x)
  \end{align*}
and 
  \begin{align*}
    \Phi_{n+1}:i_n(A_n)+I_n& \to\mathcal{B}(H), \\ i_n(x)+j&\mapsto \Phi_n(x)
  \end{align*}
are u.c.c. homomorphisms. We clearly have that $\Phi_{n+1}$ dilates $\pi_{n+1}$, and
  \[
    \pi_{n+1}\circ i_{n}=\pi_{n},\qquad \Phi_n\circ i_{n}=\Phi_{n}.
  \]
Since $A_0$ is non-selfadjoint, there exist u.c.c. homomorphisms $\pi_0$ and $\Phi_0$ on $A_0$ such that $\pi_0$ is non-trivially dilated by $\Phi_0$. Recursively this yields u.c.c. homomorphisms $(\pi_n)_n$ and $(\Phi_n)_n$ such that $\pi_n$ is dilated by $\Phi_n$ and 
  \[
    \pi_{n+1}\circ i_{n}=\pi_{n},\qquad \Phi_{n+1}\circ i_{n}=\Phi_{n}.
  \]
Hence by Lemma \ref{lemma: direct system lim homomorphisms} there exist u.c.c. homomorphisms $\pi$ and $\Phi$ on $A$ such that
  \[
    \pi\circ \lambda_n=\pi_n, \qquad \Phi\circ \lambda_n=\Phi_n.
  \]
Since $\pi_0=\pi\circ\lambda_0$ is non-trivially dilated by $\Phi_0=\Phi\circ\lambda_0$, we have that $\Phi$ is a non-trivial dilation of $\pi$. Thus $A$ is non-selfadjoint.
\end{proof}
\ \\
While it seems to be hard to compute the universal operator algebra constructed in the above proof--since one has to compute infinitely many maximal $C^*$-covers--it actually is not. The maximal $C^*$-covers are all closely related to the maximal $C^*$-cover of the initial operator algebra, as the next theorem will show.

\begin{theorem}\label{Cmax_A+I}
   Let $A$ be a non-selfadjoint operator algebra seen as a subalgebra of the maximal $C^*$-cover $(C^*_{max}(A),i_{max})$ and $I \subseteq C^*_{\max}(A)$ the Shilov ideal of $A$. Let 
    \begin{equation*}
        \mathcal{A}=\{(a+i,a+j)\ |\ a\in C^*_{max}(A), i,j \in I\}.
    \end{equation*}
    Then
    \begin{equation*}
    j:A+I\to \mathcal{A}, \ a+i\mapsto (a+i,a)
    \end{equation*}
    is well-defined, $(\mathcal{A},j)$ is the maximal $C^*$-cover of $A+I$, and $0\oplus I$ is the Shilov ideal of $j(A+I)$.
\end{theorem}

\begin{proof}
    We start by showing that $\mathcal{A}$ is indeed a $C^*$-algebra. The only thing that might not be clear is the completeness. For this, we will show that $\mathcal{A}$ is closed in $C^*_{max}(A)\oplus C^*_{max}(A)$. First note that 
      \[
        \tilde{\mathcal{A}}=\{(a,a) |\ a\in C^*_{max}(A)\}
      \]
    is a $C^*$-subalgebra of $C^*_{max}(A)\oplus C^*_{max}(A)$, and that $0\oplus I$ is an ideal in $C^*_{max}(A)\oplus C^*_{max}(A)$. Observe that the quotient map by $0\oplus I$ is isometric on $\tilde{\mathcal{A}}$ since
      \[
        \|(a,a+i)\|=\max\{\|a\|,\|a+i\|\}\ge\|a\|.
      \]
    By Lemma \ref{lem:isometric_quotient_closed}, we obtain that $\tilde{\mathcal{A}}+0\oplus I$ is closed. Since
      \[
        (a+i,a+j)=(a+i,a+i)+(0,j-i),
      \]
    we see that $\mathcal{A}=\tilde{\mathcal{A}}+0\oplus I$, and so $\mathcal{A}$ is closed.
    
    Next, notice that since $A\cap I=\{0\}$, every element in $x\in A+I$ has a unique decomposition $x=a+i$ for a $a\in A$, $i\in I$. Thus the map 
      \begin{equation*}
        j:A+I\to \mathcal{A}, a+i\mapsto (a+i,a)
      \end{equation*}
    is well-defined. It is clear that $j$ is a unital homomorphism and $C^*(j(A+I))=\mathcal{A}$. So, to see that $(\mathcal{A},j)$ is a $C^*$-cover of $A+I$, it remains to show that $j$ is completely isometric. Given a matrix $M=(a_{k,l}+i_{k,l})_{1\le k,l\le n}$, we have that
      \[
        \|M\|\ge\|(a_{k,l})_{1\le k,l\le n}\|.
      \]
   Hence
      \[
        \|j(M)\|=\max\{\|M\|,\|(a_{k,l})_{1\le k,l\le n}\|\}=\|M\|,
      \]
    and hence $j$ is completely isometric.
    
    Finally we have to show that $(\mathcal{A},j)$ is the maximal $C^*$-cover of $A+I$. By Lemma \ref{lem:uni property}, it is enough to show that every u.c.c. homomorphism on $j(A+I)$ extends to a $*$-homomorphism on $\mathcal{A}$. So let $\pi$ be a u.c.c. homomorphism on $j(A+I)$.
    
    First observe that since $j(I)=I\oplus 0$ is an ideal in $\mathcal{A}$, Theorem \ref{ucc Decom} implies that $\pi$ decomposes into $\pi_1\oplus \pi_2$ such that $\pi_1(I\oplus 0)=0$ and $\pi_2$ extends to a $*$-homomorphism on $\mathcal{A}$. Thus we only have to show that $\pi_1$ extends to a $*$-homomorphism on $\mathcal{A}$.

    Define a u.c.c. homomorphism on $A$ by
      \[
        a\mapsto \pi_1((a,a)).
      \]
    By the universal property of $C^*_{max}(A)$, this u.c.c. homomorphism has an extension to a unital $*$-homomorphism $\Pi_1$. Define a unital $*$-homomorphism $\rho$ on $\mathcal{A}$ by 
    \[
      (a+i,a+j)\mapsto \Pi_1(a+j).
    \]
    Since
   \[
     \rho((a+i,a))=\Pi_1(a)=\pi_1((a,a))=\pi_1((a+i,a))
   \]
    we see that $\rho$ extends $\pi_1$, and so $(\mathcal{A},j)$ is the maximal $C^*$-cover of $A+I$.
    
    Finally, let $p:\mathcal{A}\to C^*_{max}(A), (a+i,a+j)\mapsto a+i$. Then $(C^*_{max}(A),p|_{A+I})$ is the $C^*$-envelope of $A+I$ by Proposition \ref{prop:shilov_extension}  and thus by Proposition \ref{Prop:Equi Lattice Ideals} the Shilov ideal is given by
    \[
      \ker(p)=0\oplus I.\qedhere
    \]
\end{proof}

This characterization of $C^*_{max}(A+I)$ gives rise to a new construction of a non-selfadjoint operator algebra with only one $C^*$-cover out of an arbitrary non-selfadjoint operator algebra. In Remark \ref{Remark: KW=A+I}, we will show that this new operator algebra coincides with the one from Theorem \ref{KW construction}, which makes the following proof redundant, however, the proof is different from the one given for Theorem \ref{KW construction} and yields more insight on why the constructed operator algebra has only one $C^*$-cover.
      
\begin{theorem}\label{Thm:A+c(I)}
    Let $A$ be a non-selfadjoint operator algebra seen as a subalgebra of the maximal $C^*$-cover $(C^*_{max}(A),i_{max})$ and let $I$ be the Shilov ideal of $A$. Denote by $c_0(I)$ the $c_0$-sum of $(I)_1^\infty$ and embed $A$ into the $\ell^\infty$-sum of countably many copies of $C^*_{max}(A)$ via
      \begin{equation*}
          i(a)=(a)_{n=1}^\infty.
      \end{equation*}
    Then 
      \begin{equation*}
         i(A)+c_0(I)
      \end{equation*}
    is a non-selfadjoint operator algebra with a one point lattice given by $i(C^*_{max}(A))+c_0(I)$ and the identity embedding.
\end{theorem}

\begin{proof}
 That $i(C^*_{max}(A))+c_0(I)$ is a $C^*$-algebra follows from the observation that $c_0(I)$ is an ideal in $l^\infty(C^*_{max}(A)$, the quotient map by this ideal is isometric on $i(A)$, and Lemma \ref{lem:isometric_quotient_closed}. Analogously we obtain that $i(A)+c_0(I)$ is closed, so $i(A)+c_0(I)$ is indeed an operator algebra. Furthermore, it is straightforward to see that $C^*(i(A)+c_0(I))=i(C^*_{max}(A))+c_0(I)$, so $i(C^*_{max}(A))+c_0(I)$ with the identity embedding is indeed a $C^*$-cover of $i(A)+c_0(I)$.
 
 Next we have to show that, up to equivalence, this $C^*$-cover is unique. Let $J$ be the Shilov ideal with respect to this $C^*$-cover and $x\in J$. Then by Lemma \ref{lem:isometric_quotient_closed}, $c_0(I)+J$ is closed, and clearly an ideal. If $c_0(I)+J$ is strictly bigger than $c_0(I)$, there would exist a $0\neq a\in C^*_{max}(A)$ such that $i(a)\in c_0(I)+J$. Let $\tilde J$ be the ideal generated by $a$ in $C^*_{max}(A)$. Then $i(\tilde J)\subset c_0(I)+J$, and the quotient map by $\tilde J$ on $C^*_{max}(A)$ would be completely isometric on $A$. Hence $\tilde J \subset I$ and, in particular, $a\in I$. By assumption, there is a $b\in c_0(I)$ such that $i(a)+b\in J$. Let $(b_\lambda)$ be an approximate identity for $I$. Let $(b^{(n)}_\lambda)$ be the element in $c_0(I)$ that has $b_\lambda$ at its $n$'th component and $0$ elsewhere. Then
  \[
    (i(a)+b)b^{(n)}_\lambda=0
  \]
since it is in $c_0(I)\cap J$. But since $(bb^{(n)}_\lambda)$ converges to the element that has the $n$th entry of $b$ at its $n$th component and $0$ elsewhere, and $(i(a)b^{(n)}_\lambda)$ converges to the element that has its $n$th component of $i(a)$ at its $n$th component and $0$ elsewhere, we conclude that $i(a)=-b$. But $i(a)=(a)_1^\infty$ and $b\in c_0(I)$, so this is only possible when $a=0$, contradicting the assumption $a\neq0$ and implying $J\subset c_0(I)$. Since $c_0(I)\subset i(A)+c_0(I)$, we conclude that $J=\{0\}$. Therefore $i(C^*_{max}(A))+c_0(I)$ is the C*-envelope.

It remains to show that the maximal $C^*$-cover is given by $i(C^*_{max}(A))+c_0(I)$. For this it is enough to show that every u.c.c. homomorphism on $i(A)+c_0(I)$ extends to a unital $*$-homomorphism on $C^*(i(A)+c_0(I))$ by Lemma \ref{lem:uni property}. By Theorem \ref{ucc Decom}, we have that every u.c.c. homomorphism $\pi$ on $i(A)+c_0(I)$ decomposes into $\pi_1\oplus \pi_2$ such that $\pi_1(c_0(I))=0$ and $\pi_2$ extends to a unital $*$-homomorphism on $C^*(i(A)+c_0(I))$. So we only have to show that $\pi_1$ extends to a unital $*$-homomorphism.

Since $i$ is a injective unital $*$-homomorphism, we have that $(i(C^*_{max}(A),i\circ i_{max})$ is again a maximal $C^*$-cover of $A$. Thus by the universal property of the maximal $C^*$-cover, $\pi_1|_{i(A)}$ extends to a unital $*$-homomorphism $\Pi_1$ on $i(C^*_{max}(A))$. Define unital $*$-homomorphisms on $i(C^*_{max}(A))+c_0(I)$ by
  \[
    \rho_n(i(a)+(i_k))=\Pi_1(a+i_n).
  \]
We are done if we show that $(\rho_n)_n$ is a pointwise Cauchy sequence with respect to the operator norm. For all $a\in C^*_{max}(A), (i_k)\in c_0(I)$, it holds  that
  \[
    \|\rho_n(i(a)+(i_k))-\rho_m(i(a)-(i_k))\|=\|\Pi_1(i_n-i_m)\|\le\|i_n-i_m\|,
  \]
and since $(i_k)\in c_0(I)$, for every $\epsilon>0$ there exists a $N\in\mathbb{N}$ such that $\|i_n-i_m\|<\epsilon$ for all $n,m\ge N$, completing the proof.
\end{proof}

\begin{remark}\label{Remark: KW=A+I}
Thus far, we have seen two general constructions for a non-selfadjoint operator algebra with one $C^*$-cover, Theorem \ref{KW construction} and Theorem \ref{Thm:A+c(I)}. Of course, this raises the question of what the relation between these two is. It turns out that they are the same in the sense that there exists a surjective u.c.i. homomorphism between these two. To see this let $A_0$ be a non-selfadjoint operator algebra and let $(A_n)_n$ be the sequence of operator algebras from Theorem \ref{KW construction} with the u.c.i. homomorphisms $i_n:A_n\to A_{n+1}$, and let $A=i(A_0)+c_0(I)$ be the operator algebra from Theorem \ref{Thm:A+c(I)}. To show that $A$ coincides with the inductive limit of $(A_n)_n$, it suffices to show by Proposition \ref{Prop:Direct systems} that there exist u.c.i. homomorphisms $\lambda_n:A_n\to A$ such that
  \begin{enumerate}
      \item $\lambda_n=\lambda_{n+1}i_n$,
      \item $A=\overline{\bigcup_{n=0}^\infty \lambda_n(A_n)}.$
  \end{enumerate}
First define $\lambda_0$ by
  \[
    \lambda_0(a)=i(a)=(a,a,a,\dots).
  \]
By Theorem \ref{Cmax_A+I}, we can identify $A_1$ with 
  \[
    \{(a+i,a) \ | \ a\in A_0, i\in I\},
  \]
and the maximal $C^*$-cover of $A_1$ with
  \[
    \{(a+i,a+j) \ | \ a\in C^*_{max}(A_0), i, j\in I\}.
  \]
Then the Shilov ideal is given by $0\oplus I$. Recursively, we obtain that $A_n$ can be identified with
  \[
    \{(a+i_1,\dots,a+i_n,a)\ | \ a\in A_0, i_1,\dots,i_{n}\in I\},
  \]
and the maximal $C^*$-cover of $A_n$ with
  \[
    \{(a+i_1,\dots,a+i_n,a+i_{n+1})\ | \ a\in C^*_{max}(A_0), i_1,\dots,i_{n+1}\in I\}
  \]
with Shilov ideal $0\oplus\dots0\oplus I$. Define $\lambda_n$ by
  \[
    \lambda_n((a+i_1,\dots,a+i_{n},a))=(a+i_1,\dots,a+i_{n},a,a,a,\dots).
  \]
Noting that with the above identifications, $i_n$ is given by
  \[
    i_n((a+i_1,\dots,a+i_n,a))=(a+i_1,\dots,a+i_n,a,a),
  \]
it is easy to verify that the $\lambda_n$ fulfill $(i)$ and $(ii)$. Thus we obtain the surjective u.c.i. homomorphism by Proposition \ref{Prop:Direct systems} 
\end{remark}

We will finish this chapter with another example of an operator algebra with only one $C^*$-cover for which the relation to Theorem \ref{KW construction} is not yet clear.

Let $C^*(1,x)$ be the unital universal C*-algebra generated by a contraction $x$. For the rest of the chapter denote by $A$ the operator algebra generated by $1,x,(x^2)^*,(x^3)^*$. We will show that the C*-envelope and the maximal C*-algebra of $A$ is given by $C^*(1,x)$ with the identity embedding, and that $A$ is non-selfadjoint. One should note that $A$ is RFD, since $C^*(1,x)$ is a RFD C*-algebra.

\begin{remark}\label{rem:uniquely_determined}
Let $\pi:A\to \mathcal{B}(H)$ be a u.c.c. homomorphism. Because $\pi$ extends to a completely positive map on $A+A^\ast$, it is true that if $a\in A$ also satisfies $a^*\in A$, then $\pi(a^*)=\pi(a)^*$. Thus $\pi((x^2)^*)=\pi(x^2)^*=(\pi(x)^2)^*$ and $\pi((x^3)^*)=(\pi(x)^3)^*$. In particular, every u.c.c. homomorphism on $A$ is uniquely determined by its value on $x$. Consequently, for every contraction $T\in \B(H)$, there exists exactly one u.c.c. homomorphism $\pi:A\to \B(H)$ such that $\pi(x)=T$.
\end{remark}

\begin{theorem}\label{thm:invertible_maximal}
Let $T\in\mathcal{B}(H)$ be an invertible contraction and let $\pi:C^*(x)\to\mathcal{B}(H)$ be the unital *-homomorphism with $\pi(x)=T$. Then $\pi|_A$ is maximal.
\end{theorem}

\begin{proof}
Let $\Pi:C^*(x)\to\mathcal{B}(K)$ be a unital *-homomorpism with $H\subset K$ and $(P_H\Pi|_H)|_A=\pi|_A$. By Sarason's Lemma \cite{Sarason} we have that $H$ semi-invariant for $\Pi$. Therefore there are closed subspaces $M\subseteq L\subseteq K$ with $L\ominus M=H$ such that $M$ and $L$ are invariant for $\Pi|_A$. Since the C*-algebra $C^*(1,x^2,x^3)$ is contained in $A$, $M$ and $L$ already have to reduce $\Pi|_{C^*(1,x^2,x^3)}$. Let us write $\Pi(x)$ in matrix form 
  \begin{equation*}
    \Pi(x)=
      \begin{pmatrix}
      T_{1,1} & T_{1,2} & T_{1,3} \\
      0 & T & T_{2,3}\\
      0 & 0 & T_{3,3}
    \end{pmatrix}
  \end{equation*}
with respect to the orthogonal decomposition $K=M\oplus H \oplus (K\ominus L)$. 

Now we can calculate $\Pi(x^3)$ in three different ways. On the one hand
  \begin{equation*}
    \Pi(x^3)=
      \begin{pmatrix}
      T_{1,1}^3 & 0 & 0\\
      0 & T^3 & 0\\
      0 & 0 & T_{3,3}^3
    \end{pmatrix},
  \end{equation*}
since $x^3\in C^*(1,x^2,x^3)$ and hence $M, L$ are reducing. On the other hand
  \begin{equation*}
    \Pi(x^3)=\Pi(x)\Pi(x^2)=
      \begin{pmatrix}
      T_{1,1} & T_{1,2} & T_{1,3} \\
      0 & T & T_{2,3}\\
      0 & 0 & T_{3,3}
    \end{pmatrix}
    \begin{pmatrix}
      T_{1,1}^2 & 0 & 0 \\
      0 & T^2 & 0\\
      0 & 0 & T_{3,3}^2
    \end{pmatrix}
  =
    \begin{pmatrix}
      T_{1,1}^3 & T_{1,2}T^2 & T_{1,3}T_{3,3}^2 \\
      0 & T^3 & T_{2,3}T_{3,3}^2\\
      0 & 0 & T_{3,3}^3
    \end{pmatrix}
  \end{equation*}
where we used that $x^2\in C^*(1,x^2,x^3)$. So $T_{1,2}T^2=0$ and since $T$ is invertible by assumption , also $T_{1,2}=0$. Likewise, we receive from
  \begin{equation*}
    \Pi(x^3)=\Pi(x^2)\Pi(x)=
      \begin{pmatrix}
       T_{1,1}^2 & 0 & 0 \\
       0 & T^2 & 0\\
       0 & 0 & T_{3,3}^2
    \end{pmatrix}
    \begin{pmatrix}
      T_{1,1} & T_{1,2} & T_{1,3} \\
      0 & T & T_{2,3}\\
      0 & 0 & T_{3,3}
    \end{pmatrix}
  =
    \begin{pmatrix}
      T_{1,1}^3 & T_{1,1}^2T_{1,2} & T_{1,1}^2T_{1,3} \\
      0 & T^3 & T^2T_{2,3}\\
      0 & 0 & T_{3,3}^3
    \end{pmatrix}
  \end{equation*}
that $T_{2,3}=0$. This shows that $H$ is reducing for $\Pi$ and hence $\pi$ is maximal.
\end{proof}

\begin{theorem}\label{thm:onepointlattice}
The operator algebra $A$ is non-selfadjoint and has only, up to equivalence, one $C^*$-cover.
\end{theorem}

\begin{proof}
It is clear that $\Cmax(A)$ is given by $C^*(x)$, since every contractive operator $T\in\B(H)$ defines a u.c.c. homomorphism $\pi:A\to\B(H), x\mapsto T$. By Remark \ref{rem:uniquely_determined} every u.c.c. homomorphism $\pi:A\to\B(H)$ is uniquely determined by the contraction $\pi(x)$, and so extends uniquely to a $\ast$-homomorphism of the universal C*-algebra $C^\ast(x)$.

Let $\iota:A\to \Ce(A)$ be the embedding of $A$ into its C*-envelope, and let $\sigma:C^\ast(A)\to \Ce(A)$ be the unique morphism of C*-covers of $A$ satisfying $\sigma\vert_A=\iota$. Any invertible contraction $T\in M_n$ determines a maximal representation of $A$ by Theorem \ref{thm:invertible_maximal}, and so the $\ast$-homomorphism $C^\ast(x)\to M_n$, $x\mapsto T$ factors through $\sigma$. Since the invertible contractions are dense in the set of all contractions in $M_n$, such homomorphisms $C^\ast(x)\to M_n$ are point *-SOT dense in all homomorphisms $C^\ast(x)\to M_n$. By Lemma \ref{lem:maximal_SOT_limit}, it follows that \emph{every} $\ast$-homomorphism $C^\ast(x)\to M_n$ factors through $\sigma$. Because $C^\ast(x)$ is RFD by \cite[Theorem 5.1]{CourtneySherman}, the set of finite dimensional representations norms $C^\ast(x)$ and so $\sigma$ is isometric.

It remains to show that $A$ is non-selfadjoint. For this it is enough to show that there is at least one u.c.c. homomorphism on $A$ that is not maximal. Consider the following u.c.c. homomorphisms $\pi:A\to\mathbb{C}, x\mapsto 0$ and
  \begin{equation*}
    \psi:A\to M_2, \quad x\mapsto \begin{pmatrix} 0 & 1\\ 0 & 0\end{pmatrix}.
  \end{equation*}
Because $\psi((x^*)^2)=\psi((x^*)^3)=0$, we see that $\mathbb{C}\oplus 0$ is an invariant subspace for $\psi$, and the compression of $\psi$ onto $\mathbb{C}\oplus0\cong\mathbb{C}$ is $\pi$.  This shows that $\pi$ is not maximal and hence $A$ is non-selfadjoint.
\end{proof}

The class of non-selfadjoint operator algebras with only one C*-cover has a richer structure than one might expect and offers an opportunity for further investigations. With regard to this, we conclude our contribution to non-selfadjoint operator algebras with a one point lattice with the following question:

\begin{question}
    Is there a non-selfadjoint commutative operator algebra with a one point lattice?
\end{question}

\section{Extension of an operator algebra by a Shilov ideal}\label{Section:Extension of an operator algebra by a Shilov ideal}

Throughout this section, let $A$ be a unital operator algebra concretely embedded in a C*-algebra generated by $A$, which we will simply denote $C^\ast(A)$. Except when specifically stated, we make no assumption that $C^\ast(A)$ is the C*-envelope $C^\ast_e(A)$ or the maximal C*-algebra $C^\ast_{\max}(A)$ of $A$. Let $I\triangleleft C^\ast(A)$ be the Shilov ideal for $A$ in $C^\ast(A)$.

In Proposition \ref{prop:shilov_extension}, we showed that when enlarging $A$ to the operator algebra $A+I$, the C*-cover $C^\ast(A)$ becomes the C*-envelope. One might wonder what the rest of the C*-cover lattice $\cstarlattice(A+I)$ of $A+I$ is, and how it is related to the lattice for $A$. To explain what follows, the following perspective is useful: The C*-envelope $C^\ast_e(A+I)=C^\ast(A)$ is an extension of the C*-envelope $C^\ast_e(A)\cong C^\ast(A)/I$ by a copy of the ideal $I$, in the usual sense that there is an exact sequence of $\ast$-homomorphisms
\[
0\to I \to C^\ast(A) \to C_e^\ast(A)\to 0.
\]
In the main result of this section, Theorem \ref{thm:shilov_extension_lattice_isomorphism}, we will show that \emph{every} C*-cover of $A+I$ is uniquely built as an extension of a C*-cover of $A$ by a copy of $I$ embedded as a boundary ideal for $A$. This correspondence is an order isomorphism between the C*-cover lattices $\cstarlattice(A+I)$ and $\cstarlattice(A)$.

Here, we will summarize all the argument that follows. Given a C*-cover $(D,\eta)$ for $A+I$, there are two natural ways to reduce this to a C*-cover of $A$. Firstly, the embedded copy $\eta(I)\subseteq D$ of $I$ is a boundary ideal for $\eta(A)$ in $D$, and so there is a \textbf{quotient} C*-cover
\[
Q(D,\eta) :=
\left(\frac{D}{\eta(I)},q\eta_A\right),
\]
where $q:D\to D/\eta(I)$ is the quotient map.
Secondly, simply restricting $\eta$ to $A\subseteq A+I$ gives the \textbf{restriction} of the C*-cover
\[
R(D,\eta) :=
(C^\ast(\eta(A)),\eta\vert_A)
\]
to $A$. Conversely, given a C*-cover $(C,\iota)$ for $A$, how can one build a C*-cover for $A+I$? The most natural way is to pull back the representation $\iota$ along the projection map $p:A+I\to A$. This will not be completely isometric unless $I=0$, and so to get a complete isometry, we form a direct sum with the envelope representation $\id:A+I\to C^\ast(A)=C^\ast_e(A+I)$. This produces what we will call the \textbf{induction}
\[
N(C,\iota) =
(C^\ast((\iota q)\oplus \id_{A+I}, (\iota q)\oplus \id_{A+I})
\]
of the C*-cover $(C,\iota)$ to a C*-cover for $A+I$. Quotients, induction, and restriction induce order-preserving maps of the C*-cover lattices summarized in the diagram
\[\begin{tikzcd}
    \cstarlattice(A+I) \arrow[rr,bend left,"Q"] \arrow[rr,bend right, swap,"R"] & & \cstarlattice(A). \arrow[ll,swap,"N"]
\end{tikzcd}\]
In what follows, we prove that $Q$, $N$, and $R$ are indeed well-defined order preserving maps (Propositions \ref{prop:Q_map}, \ref{prop:D_map}, and \ref{prop:N_map}). Then, we prove that $Q$ and $N$ are mutual inverses (Theorem \ref{thm:shilov_extension_lattice_isomorphism}), and so $A+I$ and $A$ are lattice isomorphic. The map $R$ is not surjective, because it only produces C*-covers of $A$ that dominate the original cover $(C^\ast(A),\id)$, but for such covers, it restricts to a left inverse for $N$ (Proposition \ref{prop:R_vs_N}).

\begin{proposition}\label{prop:Q_map}
Let $A\subseteq C^\ast(A)$ be an operator algebra with Shilov ideal $I\triangleleft C^\ast(A)$. Let $(D,\eta)$ be a C*-cover of the operator algebra $A+I$. Then, $\eta(I)$ is a (selfadjoint) ideal in $D$, and a boundary ideal for $\eta(A)$. The assignment
\[
Q(D,\eta) =
(\frac{D}{\eta(I)},q\eta\vert_A)
\]
descends to an order preserving map
\[
\cstarlattice(A+I)\to \cstarlattice(A),
\]
which we denote by the same symbol.
\end{proposition}

\begin{proof}
First, suppose that $(D,\eta)$ is a C*-cover for $A+I$. The ideal $I\subseteq C^\ast(A)$ is itself a C*-algebra, and so the completely isometric homomorphism $\eta\vert_I$ is in fact an injective $\ast$-homomorphism, with range $\eta(I)$ that is a C*-subalgebra of $D$. In fact, $\eta(I)$ is an ideal. To prove this, let
\[
M =
\{x\in D\;\mid\; x \eta(I)\subseteq \eta(I) \text{ and } \eta(I) x\subseteq \eta(I)\}.
\]
It is straightforward to check that $M$ is a subalgebra, that because $\eta(I)$ is closed, so is $M$, and that because $\eta(I)$ is self-adjoint, so is $M$. Therefore $M$ is a C*-subalgebra. Because $I$ is an ideal in $A+I$, and $\eta$ is a homomorphism, it follows that $M$ contains $\eta(A+I)$. Since $D=C^\ast(\eta(A+I))$ is generated by $\eta(A+I)$, we must have $M=D$, so $\eta(I)$ is an ideal.

Now, we prove that $\eta(I)$ is a boundary ideal for the subalgebra $\eta(A)$. By Proposition \ref{prop:shilov_extension}, the identity representation $A+I\subseteq C^\ast(A)$ is the C*-envelope for $A$. Therefore, there is a $\ast$-homomorphism $\pi:D\to C^\ast(A)$ satisfying $\pi \eta = \id_{A+I}$. Now, let $\sigma:C^\ast(A)\to C^\ast(A)/I\cong C_e^\ast(A)$ be the quotient map by $I$ to the C*-envelope of $A$, so we have $\ast$-homomorphisms
\[\begin{tikzcd}
    D \arrow[r,"\pi"] & C^\ast(A) \arrow[r,"\sigma"] & C_e^\ast(A)
\end{tikzcd}\]
for which the composition $\sigma\pi$ is completely isometric on $\eta(A)$. Since $\ker\sigma = I$ and $\pi\eta = \id_{A+I}$, we have $\ker \sigma\pi \supseteq \eta(I)$. Therefore $\pi\sigma$ factors through a $\ast$-homomorphism $D/\eta(I)\to C_e^\ast(A)$, making the diagram
\[\begin{tikzcd}
    D \arrow[r,"\pi"] \arrow[d,"q"] & C^\ast(A) \arrow[r,"\sigma"] & C_e^\ast(A) \\
    \frac{D}{\eta(I)} \arrow[urr,swap]
\end{tikzcd}\]
commute. Because $\sigma\pi$ is completely isometric on $\eta(A)$, and the diagonal map is completely contractive, the quotient map $q$ is completely isometric on $\eta(A)$. Therefore $q\eta\vert_A$ is completely isometric. Note also that because $D=C^\ast(\eta(A+I))$, then $D/\eta(I) =C^\ast(q\eta(A+I))=C^\ast(q\eta(A))$ is generated by $q\eta(A)$, so that $(D/\eta(I),q\eta\vert_A)$ is indeed a C*-cover of $A$.

Now, suppose that $(D,\eta)$ and $(E,\rho)$ are C*-covers of $A+I$ for which $(D,\eta)\ge (E,\rho)$. Then, there is a $\ast$-homomorphism $\pi:D\to E$ satisfying $\pi\eta = \rho$. Then $\pi(\eta(I)) = \rho(I)$, and so $\pi$ induces a $\ast$-homomorphism
\[
\tilde{\pi}:\frac{D}{\eta(I)}\to \frac{E}{\rho(I)}
\]
satisfying $\tilde{\pi}(x+\eta(I)) = \pi(x)+\rho(I)$. The map $\tilde{\pi}$ is a morphism of C*-covers $(D/\eta(I),q\eta\vert_A)\to (E/\rho(I),q\rho\vert_A)$, and so $Q(D,\eta)\ge Q(E,\rho)$. Therefore $Q$ induces a well-defined order preserving map $\cstarlattice(A+I)\to \cstarlattice(A)$.
\end{proof}

\begin{proposition}\label{prop:D_map}
Let $A\subseteq C^\ast(A)$ be an operator algebra with Shilov ideal $I\triangleleft C^\ast(A)$. For C*-covers $(D,\eta)$ of the operator algebra $A+I$, the assignment
\[
R(D,\eta) =
(C^\ast(\eta(A)),\eta\vert_A)
\]
descends to an order preserving map
\[
\cstarlattice(A+I)\to \cstarlattice(A),
\]
which we denote by the same symbol. And, for any such C*-cover $(D,\eta)$, we have $R(D,\eta)\ge (C^\ast(A),\id)$ in the ordering on C*-covers of $A$.
\end{proposition}

\begin{proof}
If $(D,\eta)$ is a C*-cover for $A$, then the restriction $\eta\vert_A$ remains completely isometric, and so $R(D,\eta) = (C^\ast(\eta(A)),\eta\vert_A)$ is a C*-cover of $A$.

Suppose that $(D,\eta)\ge (E,\rho)$ as C*-covers of $A+I$, so that there is an $\ast$-homomorphism $\pi:D\to E$ with $\pi\eta = \rho$. Then $\pi\vert_{C^\ast(\eta(A))}:C^\ast(\eta(A))\to C^\ast(\rho(A))$ is a morphism of C*-covers $R(D,\eta)\to (E,\rho)$. Therefore $Q$ induces a well-defined order preserving map $\cstarlattice(A+I)\to \cstarlattice(A)$.

Now, by Proposition \ref{prop:shilov_extension}, the C*-cover $A+I\subseteq C^\ast(A)$ is the C*-envelope for $A+I$. Therefore there is an $\ast$-homomorphism $\pi:D\to C^\ast(A)$ satisfying $\pi\eta = \id$. Therefore the restriction $\pi\vert_{C^\ast(\eta(A))}:C^\ast(\eta(A))\to C^\ast(A)$ is a morphism of C*-covers $R(D,\eta)\to (C^\ast(A),\iota)$, so $R[D,\eta]\ge [C^\ast(A),\iota]$ in the ordering.
\end{proof}

\begin{proposition}\label{prop:N_map}
Let $A\subseteq C^\ast(A)$ be an operator algebra with Shilov ideal $I\triangleleft C^\ast(A)$. For C*-covers $(C,\iota)$ of the operator algebra $A+I$, the assignment
\[
N(C,\iota) =
(C^\ast((\iota p)\oplus \id_{A+I}),(\iota p)\oplus \id_{A+I}),
\]
where $p:A+I\to A$ is the completely contractive homomorphism projecting to $A$, descends to an order preserving map
\[
\cstarlattice(A)\to \cstarlattice(A+I),
\]
which we denote by the same symbol.
\end{proposition}

\begin{proof}
Given a C*-cover $(C,\iota)$ of $A$,
\[
N(C,\iota) :=
(C^\ast(\iota p \oplus \id),\iota p \oplus \id)
\]
is a C*-cover for $A+I$. Indeed, $\iota$ and $p$ are completely contractive homomorphisms, and therefore so too is $\iota p$. Since $p$ zeroes out $I$, the map $\iota p$ is certainly not completely isometric on $A+I$, but this is fixed by direct summing with the identity representation $\id:A+I\subseteq C^\ast(A)$.

Now, suppose that $(C,\iota)\ge (D,\eta)$ as C*-covers of $A$, so that there is an $\ast$-homomorphism $\pi:C\to D$ satisfying $\pi\iota = \eta$. Then
\[
(\pi\oplus \id):C\oplus C^\ast(A)\to D\oplus C^\ast(A)
\]
is a $\ast$-homomorphism satisfying $(\pi\oplus \id)((\iota p)\oplus \id) =(\pi\iota p)\oplus \id =(\eta p)\oplus \id$. This restricts to a morphism of C*-covers $N(C,\iota)\to N(D,\eta)$. Therefore $N$ induces a well-defined order preserving map $\cstarlattice(A)\to \cstarlattice(A+I)$.
\end{proof}

\begin{theorem}\label{thm:shilov_extension_lattice_isomorphism}
Let $A\subseteq C^\ast(A)$ be an operator algebra with Shilov ideal $I\triangleleft C^\ast(A)$, and consider the larger operator algebra $A+I$. The quotient and induction maps
\[\begin{tikzcd}
    \cstarlattice(A+I) \arrow[rr,bend left,"Q"] && \cstarlattice(A) \arrow[ll,bend left,"N"]
\end{tikzcd}\]
are order-preserving mutual inverses, and so $A$ and $A+I$ are lattice isomorphic. In particular, every C*-cover of $A+I$ is an induction of a unique C*-cover of $A$, which is a C*-cover that is an extension of a C*-cover of $A$ by a copy of $I$, embedded as a boundary ideal for the copy of $A$.
\end{theorem}

\begin{proof}
Let $(C,\iota)$ be a C*-cover of $A$. Then let $N(C,\eta) = (C^\ast((\iota p\oplus \id)(A+I)),\iota p\oplus \id)=:(D,\eta)$, where $p:A+I\to A$ is the completely contractive projection homomorphism and $\id:A+I\to C^\ast(A)$ is the identity representation. The image of the ideal $I$ is $\eta(I) = 0\oplus I$. Because $I$ is the Shilov ideal, the map $\rho:A\to C^\ast(A)\to C^\ast(A)/I$ is a C*-envelope for $A$, and so there is an injective $\ast$-homomorphism $\pi:C\to C\oplus \frac{C^\ast(A)}{I}$ which satisfies $\pi(a) =a\oplus \rho(a)$ for all $a\in A$. Putting it all together, we have a composition of injective $\ast$-homomorphisms
\[\begin{tikzcd}
    C \arrow[r,"\pi"] & C\oplus \frac{C^\ast(A)}{I} \arrow[r,"\sim"] & \frac{C\oplus C^\ast(A)}{0\oplus I}\arrow[r,equals] & \frac{D}{\eta(I)}.
\end{tikzcd}\]
which gives an isomorphism of C*-covers $(C,\iota)\to QN(C,\iota)$ of $A$, so in terms of equivalence classes, $QN[C,\iota] = [C,\iota]$. Therefore $QN=\id$.

Now, let $(D,\eta)$ be a C*-cover of $A+I$, so that $Q(D,\eta) = (\frac{D}{\eta(I)},q\eta\vert_A)$, where $q:D\to D/\eta(I)$ is the quotient map by the $\eta(A)$-boundary ideal $\eta(I)$. Then
\[
NQ(D,\eta) =
(C^\ast(((q\eta p)\oplus \id)(A+I)),q\eta p\oplus \id)
\]
where
\[
C^\ast(q\eta p \oplus \id) \subseteq
\frac{D}{\eta(I)}\oplus C^\ast(A).
\]
By Proposition \ref{prop:shilov_extension}, the C*-algebra $C^\ast(A)$ is the C*-envelope for $A+I$, and so there is an $\ast$-homomorphism $\pi:D\to C^\ast(A)$ satisfying $\pi\eta = \id$. Therefore we have a $\ast$-homomorphism
\[
\sigma := q\oplus \pi:D\to \frac{D}{\eta(I)}\oplus C^\ast(A)
\]
that satisfies $\sigma\eta =
(q\eta)\oplus \id = (q\eta p)\oplus \id$. Note that $q\eta =q\eta p$ because $p:A+I\to A$ annihilates $I$, and $q$ annihilates $\eta(I)$. Therefore $\sigma$ is in fact a morphism of C*-covers $(D,\eta)\to NQ(D,\eta)$. Moreover, we have $\ker\sigma = \ker q \cap \ker \pi = 0$, because $\ker q = \eta(I)$, and $\pi$ is completely isometric on $\eta(I)$. Therefore $\sigma$ is injective and so gives an isomorphism of C*-covers $(D,\eta)\to NQ(D,\eta)$, so in terms of equivalence classes $NQ[D,\eta]=[D,\eta]$. Therefore $NQ=\id$, so $N$ and $Q$ are mutually inverse.
\end{proof}

\begin{corollary}
Let $A\subseteq C^\ast(A)$ be an operator algebra with Shilov ideal $I\triangleleft C^\ast(A)$. Let $(C^\ast_{\max}(A),i_A)$ be the maximal C*-cover of $A$. Then the maximal C*-cover of $A+I$ is
\[
(C^\ast_{\max}(A+I),i_{A+I}) =
(C^\ast((i_A p\oplus \id)(A+I)),i_A p\oplus \id).
\]
I.e. $C^\ast_{\max}(A+I)$ is the C*-algebra generated by the image of $A+I$ under the map
\[
a+x \mapsto 
i_A(a)\oplus (a+x)
\]
for $a\in A$ and $x\in I$. There is an exact sequence
\[\begin{tikzcd}
0 \arrow[r] & I \arrow[r,"i_{A+I}\vert_I"] & C^\ast_{\max}(A+I) \arrow[r,"\tilde{p}"] & C^\ast_{\max}(A)\arrow[r] & 0
\end{tikzcd}\]
where $\tilde{p}$ is the unique $\ast$-homomorphism that extends the completely contractive homomorphism $i_A p:A+I\to A \to C^\ast_{\max}(A)$.
\end{corollary}

Note that in the case where $I\triangleleft C^\ast_{\max}(A)$ is the Shilov ideal for $A$ embedded in its maximal C*-cover, this recovers Theorem \ref{Cmax_A+I}.

% \todo[inline]{TODO: Comment on how this is connected to Theorem \ref{Cmax_A+I}.}

\begin{proposition}\label{prop:R_vs_N}
Let $A\subseteq C^\ast(A)$ be an operator algebra with Shilov ideal $I\triangleleft C^\ast(A)$. Let $N:\cstarlattice(A)\to \cstarlattice(A+I)$ the induction map, and denote by $R:\cstarlattice(A+I)\to \cstarlattice(A)$ the restriction map. Then for any C*-cover $(C,\iota)$ of $A$, we have that
\[
RN[C,\iota] =
[C,\iota]\vee [C^\ast(A),\id]
\]
is the lattice join with the C*-cover $C^\ast(A)$ in which $A$ is concretely embedded. Furthermore, $RNR=R$, and for any C*-cover $[D,\eta]$ of $A+I$,
\[
NR[D,\eta] = [D,\eta]\vee N[C^\ast(A),\id]
\]
is the lattice join with the induction of $[C^\ast(A),\id]$.

In particular, $R$ is a surjective order preserving map onto the sub-lattice of C*-covers of $A$ that dominate $[C^\ast(A),\id]$ in the ordering.
\end{proposition}

\begin{proof}
Recall that the lattice join is the direct sum
\[
[C,\iota]\vee[C^\ast(A),\id] =
[C^\ast((\iota\oplus \id)(A)),\iota\oplus \id].
\]
We have
\[
N[C,\iota] =
[C^\ast((\iota p\oplus \id)(A+I)),\iota p \oplus \id ].
\]
Upon restricting to $A$, we have $p\vert_A = \id_A$, and so
\[
RN[C,\iota] =
[C^\ast((\iota p\oplus \id)(A)),(\iota p\oplus \id)\vert_A] =
[C^\ast((\iota \oplus \id)(A)),
\iota \oplus \id],
\]
which is exactly the lattice join $[C,\iota]\vee [C^\ast(A),\id]$.

Now, we have that $[C,\iota]\ge [C^\ast(A),\id]$ in the ordering if and only if $[C,\iota]\vee [C^\ast(A),\id]=[C,\iota]$, if and only if $RN[C,\iota]=[C,\iota]$, so $R$ maps surjectively onto the set of (classes of) C*-covers dominating $[C^\ast(A),\id]$. In Proposition \ref{prop:D_map}, we showed that for any C*-cover $[D,\eta]$ of $A+I$, that $R[D,\eta]\ge [C^\ast(A),\iota]$, and so it follows that $RNR=R$.

Now, let $(C,\iota)$ be a C*-cover of $A$. Then, because $N$ is a lattice isomorphism
\[
NRN[C,\iota] =
N([C,\iota]\vee [C^\ast(A),\id]) =
N[C,\iota]\vee N[C^\ast(A),\id].
\]
Because $N$ is surjective, this proves that for any C*-cover $[D,\eta]$ of $A+I$, that
\[
NR[D,\eta] =
[D,\eta] \vee N[C^\ast(A),\id].
\]\qedhere
\end{proof}

\begin{corollary}\label{cor:R_vs_Q}
Let $A\subseteq C^\ast(A)$ be an operator algebra with Shilov ideal $I\triangleleft C^\ast(A)$. For any C*-cover $(D,\eta)$ of $A+I$, we have $R[D,\eta] \ge Q[D,\eta]$. In fact, there is the explicit morphism of C*-covers of $A$ given by
\[\begin{tikzcd}
    C^\ast(\eta(A)) \arrow[r,hook] & D \arrow[r] & \frac{D}{\eta(I)}
\end{tikzcd}\]
\end{corollary}

\begin{proof}
We have $NR[D,\eta] = [D,\eta]\vee N[C^\ast(A),\id] \ge [D,\eta] = NQ[D,\eta]$, and so applying $Q=N^{-1}$ gives the stated ordering. Independently, it is straightforward to check that the stated composition is a morphism of C*-covers of $A$.
\end{proof}

Now, we are in a position to identify exactly when the two operations of restriction $R$ and quotient $Q$ coincide for C*-covers of $A+I$.

\begin{corollary}
Let $A\subseteq C^\ast(A)$ be an operator algebra with Shilov ideal $I\triangleleft C^\ast(A)$. Let $(D,\eta)$ be a C*-cover of $A+I$, and let $q:D\to D/\eta(I)$ be the quotient map. The following are equivalent:
\begin{itemize}
\item[(1)] $C^\ast(\eta(A))\cap \eta(I)=0$ and so $D=C^\ast(\eta(A))\oplus \eta(I)$ as a direct sum of vector spaces.
\item[(2)] The composed $\ast$-homomorphism
\[\begin{tikzcd}
    C^\ast(\eta(A)) \arrow[r,hook] & D \arrow[r,"q"] & \frac{D}{\eta(I)}
\end{tikzcd}\]
is a $\ast$-isomorphism, and so $Q[D,\eta] = R[D,\eta]$.
\item[(3)] There is a $\ast$-homomorphism
\[
\pi:D\to C^\ast(A)
\]
satisfying $\pi\eta\vert_A=\id_A$ and $\pi(\eta(I))=0$.
\item [(4)] The C*-cover $(D,\eta)$ dominates $N[C^\ast(A),\id]= [C^\ast((p\oplus \id)(A)),p\oplus \id]$.
\item[(5)] The C*-cover $(D,\eta)$ is an induction $N(C,\iota)$ of a C*-cover of $(C,\iota)$ of $A$ that dominates $(C^\ast(A),\id)$.
\end{itemize}
\end{corollary}

\begin{proof}
The equivalence of (1) and (2) uses only the following standard facts from C*-algebra theory: If $A\subseteq B$ are C*-algebras, and $I\triangleleft B$ is an ideal, then the composition
\[
A\to B\to \frac{B}{I}
\]
composes with a $\ast$-isomorphism to
\[
A\to \frac{A+I}{I} \cong
\frac{A}{A\cap I},
\]
which is an isomorphism if and only if $A\cap I=0$. If, furthermore, $A$ and $I$ generate $B$ as a C*-algebra, then it follows that $B=A\oplus I$ as vector spaces.

Item (3) is equivalent to the existence of a $\ast$-homomorphism $\tilde{\pi}:D/\eta(I)\to C^\ast(A)$ that is a morphism of covers $Q(D,\eta) \to (C^\ast(A),\id)$. Therefore (3) holds if and only if $Q[D,\eta]\ge [C^\ast(A),\id]$, which holds if and only if $[D,\eta] \ge N[C^\ast(A),\id]$, which is (4). And, (4) is equivalent to (5) by applying the lattice isomorphisms $Q$ and $N$. So, (3), (4), and (5) are equivalent.

Now, item (2) is equivalent to the assertion that the morphism in Corollary \ref{cor:R_vs_Q} is an isomorphism, which is equivalent to asserting that $Q[D,\eta] = R[D,\eta]$, and so $[D,\eta] = NQ[D,\eta] =NR[D,\eta] =[D,\eta]\vee N[C^\ast(A),\id]$, which is equivalent to (4).
\end{proof}

\section{Uncountability of the C*-cover lattice}\label{Section:Uncountability of the C*-cover lattice}

One might expect that if there is an non-selfadjoint operator algebra with only one C*-cover, there are also some with exactly two C*-covers. Surprisingly, this is not the case. More generally, we will show in this chapter that every operator algebra with at most countably many C*-covers already has exactly one C*-cover, or conversely that the lattice of C*-covers is either a singleton or uncountable. The following lemma will be of great help here. The lemma and its proof are from \cite{Kacnelson}.

\begin{lemma}\label{lem:twist_dilation}
Let $\pi:\mathcal{A}\to\mathcal{B}(H)$ be a u.c.c. homomorphism and suppose that $\sigma:\mathcal{A}\to\mathcal{B}(K)$ is a u.c.c. homomorphism that dilates $\pi$, with $H\subseteq K$. By Sarason's Lemma \cite{Sarason}, there exists a decomposition of $K$ as a direct sum $H_1\oplus H\oplus H_2$ such that
  \begin{equation*}
    \sigma=
    \begin{pmatrix}  \pi_1 & \sigma_{1,2}  & \sigma_{1,3}\\  0  & \pi  &\sigma_{2,3}\\  0   &   0   & \pi_3 \end{pmatrix}.
  \end{equation*}
Then, for any choice of $z\in\overline{\mathbb{D}}$,
  \begin{equation*}
    \sigma_z:=\begin{pmatrix}  \pi_1 & z\sigma_{1,2}  & z^2\sigma_{1,3}\\  0  & \pi  &z\sigma_{2,3}\\  0   &   0   & \pi_3 \end{pmatrix}
  \end{equation*}
is again a u.c.c. homomorphism dilating $\pi$. The map 
  \[
    \overline{\bbD}\to \text{ucc}(A,B(K)), z\mapsto \sigma_z
  \]
is norm-continuous, $\sigma_1=\sigma$, and $\sigma_0$ is a trivial dilation of $\pi$. Moreover, whenever $|z|=|w|$, the representations $\sigma_z$ and $\sigma_w$ are unitarily equivalent.
\end{lemma}

\begin{proof}
If $z=0$, then $\sigma_0$ is a direct sum of u.c.c. homomorphisms, which is still a u.c.c. homomorphism. And, $\sigma_0=\pi_1\oplus \pi\oplus \pi_2$ is a trivial dilation of $\pi$. Otherwise, suppose that $z\ne 0$. Then, the map $\sigma_z$ is obtained via the similarity
\begin{equation*}
   \sigma_z=\begin{pmatrix} 1& 0  & 0\\ 0  & z^{-1}  &0   \\ 0 & 0& z^{-2}\end{pmatrix}\begin{pmatrix}  \pi_1 & \sigma_{1,2}  & \sigma_{1,3}\\  0  & \pi  &\sigma_{2,3}\\  0   &   0   & \pi_3 \end{pmatrix}\begin{pmatrix} 1& 0  & 0\\ 0  & z  &0   \\ 0 & 0& z^2\end{pmatrix}.
  \end{equation*}
  This shows that the maps $\sigma_z$ are homomorphisms and additionally that the map 
  \begin{align*}
    \mathbb{C}&\to\mathcal{B}(K^n) \\  z&\mapsto \sigma_z(x)
  \end{align*}
is analytic for every $x\in M_n(\mathcal{A})$. So we can apply the maximum modulus principle, which yields that
  \begin{equation*}
    \sup_{z\in\overline{\mathbb{D}}}\|\sigma_z(x)\|=\sup_{z\in\mathbb{T}}\|\sigma_z(x)\|=\|\sigma_z(x)\|\le\|x\|
  \end{equation*}
for every $x\in M_n(\mathcal{A})$. Therefore $\sigma_z$ is u.c.c for all $z\in\overline{\mathbb{D}}$. 

Finally, if $|z|=|w|>0$, then $\sigma_w = U\sigma_w U^*$, where
\[
U = \begin{pmatrix}
    1 & 0 & 0 \\
    0 & zw^{-1} & 0 \\
    0 & 0 & z^2w^{-2}
\end{pmatrix}
\]
is unitary.
\end{proof}

Note that, more generally, for any nonzero $z,w\in \overline{\bbD}$, the representations $\sigma_z$ and $\sigma_w$ are similar, via the similarity $S=\diag(1,zw^{-1},z^2w^{-2})$.

\begin{example}\label{eg:Toeplitz_representation_twist_dilation}
Let $\pi:A(\bbD)\to C^\ast(S)$ be the completely isometric representation that sends the generator $f$ to the unilateral shift $S$. A maximal dilation of this representation sends $f$ to the sz. Nagy dilation of $S$ to a unitary
\[
U =
\begin{pmatrix}
    S & I-S S^\ast \\
    0 & S^\ast
\end{pmatrix}.
\]
The process in Lemma \ref{lem:twist_dilation} applied to this dilation (note that because $U$ extends $S$, we can take $H_1=0$ in Sarason's Lemma) corresponds to the family of completely isometric representations of $A(\bbD)$ that send $z$ to
\[
V_z =
\begin{pmatrix}
    S & z(I-SS^\ast) \\
    0 & S^\ast
\end{pmatrix}
\]
for $z\in \overline{\bbD}$. We will see that for $|z|\ne |w|$, the representations determined by $V_z$ and $V_w$ generate C*-covers that are not even comparable in the ordering of C*-covers, unless $|z|=1$ or $|w|=1$.

Suppose that there is a morphism of C*-covers $\pi:C^\ast(V_z)\to C^\ast(V_w)$, meaning a (unital) $\ast$-homomorphism that satisfies $\pi(V_z)=V_w$. We have
\begin{align*}
V_z^\ast V_z &=
\begin{pmatrix}
    I & 0 \\
    0 & |z|^2(I-SS^\ast) + SS^\ast
\end{pmatrix} =
\begin{pmatrix}
    I & 0 \\
    0 & I-SS^\ast
\end{pmatrix} +
|z|^2\begin{pmatrix}
    0 & 0 \\
    0 & SS^\ast
\end{pmatrix}.
\end{align*}
Since the two matrices on the right-hand-side are orthogonal projections summing to the identity, the spectrum is
\[
\sigma(V_z^\ast V_z) =
\{|z|^2,1\}.
\]
Identically, $\sigma(V_w^\ast V_w)=\{|w|^2,1\}$. Since a $\ast$-homomorphism shrinks spectra, we have
\[
\sigma(V_w^\ast V_w)=
\sigma(\pi(V_z^\ast V_z))\subseteq
\sigma(V_z^\ast V_z),
\]
and so $\{|w|^2,1\}\subseteq \{|z|^2,1\}$. Therefore, we must have $|w|=|z|$--in which case $V_z$ and $V_w$ are unitarily equivalent, or else $|w|=1$--in which case $C^\ast(V_w)=C^\ast(U)$ is the C*-envelope.
\end{example}

Example \ref{eg:Toeplitz_representation_twist_dilation} demonstrates that even though the definition of $\sigma_z$ in Lemma \ref{lem:twist_dilation} is a nice continuous map, the resulting C*-covers produced along the path from $z=0$ to $z=1$ can be ``badly discontinuous" with respect to the C*-cover ordering. In Example \ref{eg:Toeplitz_representation_twist_dilation}, the C*-covers produced start from $\sigma_0$, which generates a minimal C*-cover (the Toeplitz algebra), end at $\sigma_1$ giving the C*-envelope, but along the way the representations $\sigma_z$ for $0\le z< 1$ are all mutually incomparable in the ordering of the C*-cover lattice.

\begin{example}\label{eg:disk_algebra_twist_dilation}
Now, consider the completely isometric representation $\pi:A(\bbD)\to C(\overline{\bbD})$ given by inclusion. This is the universal C*-algebra generated by a normal contraction $f\in C(\overline{\bbD})$ with $f(z)=z$. A maximal dilation to a unitary is given by the Schaeffer dilation
\[
f\mapsto U :=
\begin{pmatrix}
    \ddots &\\
    &0 & 1 & \\
    && 0 &  \sqrt{1-|f|^2} & -\overline{f} \\  
    && 0 & f & \sqrt{1-|f|^2} & 0 \\
    && & & 0 & 1 \\
    &&&&&0&1\\
    &&&&&&&\ddots
\end{pmatrix} \in C(\overline{\bbD},B(\ell^2))
\]
For $z\in \overline{\bbD}$, the dilation produced by the process in Lemma \ref{lem:twist_dilation} sends the generator $f\in A(\bbD)$ to
\[
f \mapsto V_z :=
\begin{pmatrix}
    \ddots &\\
    &0 & 1 & \\
    && 0 &  z\sqrt{1-|f|^2} & -z^2\overline{f} \\  
    && 0 & f & z\sqrt{1-|f|^2} & 0 \\
    && & & 0 & 1 \\
    &&&&&0&1\\
    &&&&&&&\ddots
\end{pmatrix} \
\]
It is straightforward to check that when $|z|\ne 0$ and $|z|\ne 1$, the operator $V_z$ is not normal. Therefore, even though $C^\ast(V_0)=C(\overline{\bbD})$, and $C^\ast(V_1)=C(\bbT)$, for $z\in (0,1)$, none of the C*-covers $C^\ast(V_z)$ sit between $C(\overline{\bbD})$ and $C(\bbT)$ in the ordering of C*-covers.

\end{example}

It is time to proof the lattice dichotomy theorem.

\begin{theorem}
\label{thm:uncountable_lattice}
Assume that the operator algebra $A$ has more than one C*-cover, then the cardinality of the set of C*-covers is at least the continuum $\mathfrak{c}$.
\end{theorem}

\begin{proof}
Let $(C^*_{max}(A),i_{max})$ be the maximal $C^*$-cover of $A$, and $(C^*_e(A),i_e)$ the $C^*$-envelope. For every u.c.c. homomorphism $\pi:A\to\mathcal{B}(H)$, we will denote the unique unital $*$-homomorphism obtain by the universal property, see Lemma \ref{lem:uni property}, of $C^*_{max}(A)$ by $\tilde{\pi}$, i.e. a unital  $*$-homomorphism $\tilde{\pi}:C^*_{max}(A)\to\mathcal{B}(H)$ such that $\tilde{\pi}\circ i_{max}=\pi$.

Since $(\Cmax(A),i_{\max})\neq (C^*_e(A),i_e)$, the map $i_{\max}$ is not maximal by Lemma \ref{lem: maximal u.c.c. and envelope}, and so there exists a maximal dilation $\sigma$ of $i_{\max}$ by Proposition \ref{prop:existence max dilation}. Applying Sarason's lemma \cite{Sarason}, we obtain that 
\begin{equation*}
    \sigma=\begin{pmatrix}  \pi_1 & \sigma_{1,2}  & \sigma_{1,3}\\  0  & i_{\max}  &\sigma_{2,3}\\  0   &   0   & \pi_3 \end{pmatrix}.
  \end{equation*}
Let $\sigma_z$ the u.c.c. homomorphisms from Lemma \ref{lem:twist_dilation}. Because $\pi$ is completely isometric, so is each $\sigma_z$. Recall that by Lemma \ref{lem:intersection shilov ideal} the Shilov ideal of $A$ seen as a subalgebra of $\Cmax(A)$ is given by
  \begin{equation*}
    I=\bigcap\{\ker(\tilde{\rho})\,\mid\, \rho: A\to\mathcal{B}(H)\ \textup{u.c.c. and maximal} \},
  \end{equation*}
which is not trivial because $ A$ has more than one C*-cover and Lemma \ref{lem:empty Shilov}. Therefore there exists a $x\in I$ with $x\neq0$. By choice, $\sigma=\sigma_1$ is maximal and thus
  \begin{equation*}
    \tilde{\sigma}_1(x)=\tilde{\sigma}(x)=0.
  \end{equation*}
But $\tilde{i}_{\max}$ is an isometric representation of $\Cmax( A)$, so $\ker(\tilde{i}_{\max})=\{0\}$ and since $\sigma_0$ is a trivial dilation of $i_{\max}$, we obtain
\begin{equation*}
\tilde{\sigma}_0(x)\neq 0.
\end{equation*}
Without loss of generality, by re-scaling $x$ if necessary, assume that $\|\tilde{\sigma}_0(x)\|=1$. The function $\delta_y:\overline{\mathbb{D}}\to\mathbb{R}, z\mapsto \|\tilde{\sigma}_z(y)\|$ is continuous for every $y\in i_{\max}( A)$. Since $\tilde{\sigma}_z$ is a *-homomorphism, we obtain that $\delta_y$ is continuous for every $y$ in the *-algebra generated by $i_{\max}( A)$. However, contractivity of $\tilde{\sigma}_z$ and the triangle inequality imply that
   \begin{equation*}
      |\delta_y(z)-\delta_y(\omega)|\le 2\|y-v\|+|\delta_{v}(z)-\delta_{v}(\omega)|
  \end{equation*}
for every $y, v\in \Cmax( A)$. Thus $\delta_y$ is continuous for every $y\in \Cmax( A)$.

We have $\|\delta_x(0)\|=\|\tilde{\sigma}_0(x)\|=1$ and $\|\delta_x(1)\|=\|\tilde{\sigma}_1(x)\|=0$. Hence, by the Intermediate Value Theorem, for each $t\in [0,1]$, there exists $z_t\in [0,1]$ satisfying $\delta_x(z_t)=\|\tilde{\sigma}_{z_t}(x)\|=t$. We will show that the C*-covers $(C^\ast(\sigma_{z_t}(A)),\sigma_{z_t})$ are pairwise non-isomorphic for each $t\in [0,1]$.

So, let $s,t\in [0,1]$ and suppose that there is a $\ast$-isomorphism $\rho:C^\ast(\sigma_{z_s}(A))\to C^\ast(\sigma_{z_t}(A))$ satisfying $\rho\sigma_{z_s} = \sigma_{z_t}$. By uniqueness of the extensions of the homomorphisms involved to $C^\ast_{\max}(A)$, we have $\rho\tilde{\sigma}_{z_s} = \tilde{\sigma}_{z_t}$. So,
\[
s = 
\|\tilde{\sigma}_{z_s}(x)\| =
\|\rho\tilde{\sigma}_{z_s}(x)\| = 
\|\tilde{\sigma}_{z_t}(x)\| = t,
\]
so $s=t$. Contrapositively, this shows that for $s\ne t$ in $[0,1]$, the C*-covers $C^\ast(\sigma_{z_s}(A))$ are non-isomorphic.
\end{proof}

\begin{remark}
    The proof of the above theorem yields more than a mere bound on the cardinality of $\cstarlattice(A)$. Equip $\cstarlattice(A)$ with the smallest topology such that, for every $a\in C^*_{max}(A)$, the map
      \begin{align*}
        \cstarlattice(A)&\to\mathbb{R}, \\ [\mathcal{A},\pi]&\mapsto \|\tilde\pi(a)\|
      \end{align*}
    is continuous. With this topology $\cstarlattice(A)$ is Hausdorff.
    
    Given a $C^*$-cover $(\mathcal{A},\pi)$ of $A$, let $(\sigma_z)_{z\in\overline{\mathbb{D}}}$ be the u.c.i. homomorphisms from the proof of Theorem \ref{thm:uncountable_lattice}. Define
    \begin{align*}
      p:[0,1]&\to \cstarlattice(A), \\t &\mapsto [C^*((\sigma_t\oplus\pi)(A), \sigma_t\oplus \pi].
    \end{align*}
    Then $p$ is continuous with $p(0)=[C^*_{max}(A),i_{max}]$ and $p(1)=[\mathcal{A},\pi]$. Since $(\mathcal{A},\pi)$ was arbitrary, $\cstarlattice(A)$ is path-connected.
    
    To deduce the stated cardinality conclusion for $\cstarlattice(A)$, it suffices to observe that a path-connected Hausdorff space either consists of a single point or has cardinality at least $\mathfrak{c}$. This fact is well-known, for completeness we sketch a proof.
    
    Assume that a path-connected Hausdorff space $X$ has cardinality strictly greater than $1$. Then there exists a path $\gamma$ with $\gamma(0)\neq\gamma(1)$. After verifying that $\gamma([0,1])$ is a \textit{Peano space}, i.e., compact, connected, locally connected and metrizable, \cite[Theorem 31.2]{Willard} implies that $\gamma([0,1])$ is \textit{arc-connected}, i.e., for every $x\neq y\in\gamma([0,1])$ there is a injective path $\tilde{\gamma}:[0,1]\to\gamma([0,1])$ with $\tilde{\gamma}(0)=x$ and $\tilde{\gamma}(1)=y$. Since $1<|\gamma([0,1])|$, there exists such an injective path $\tilde{\gamma}$. Therefore
      \[
        \mathfrak{c}=|[0,1]|\le|\gamma([0,1])|\le|X|.
      \]
    
\end{remark}
%Contrast with Ian's paper. Is it the same topology.

\begin{theorem}\label{thm:maintheorem}
Let $A$ be a separable operator algebra. Then
    \[
      |C^*\textup{-Lat}(A)|=\begin{cases} 1 & or\\ \mathfrak{c}. & \\ \end{cases}
  \]
\end{theorem}

\begin{proof}
We know that the cardinality is at least $\mathfrak{c}$. To see that it is not bigger than $\mathfrak{c}$, note that every C*-cover correspond to a norm closed ideal in $\Cmax(\mathcal{A})$. But since $\mathcal{A}$ is separable, $\Cmax(\mathcal{A})$ is also separable. Thus $\Cmax(\mathcal{A})$ is a second countable space with respect to the norm topology and therefore the cardinality of the family of all closed sets in $\Cmax(\mathcal{A})$ is at most $\mathfrak{c}$.
\end{proof}

After examining Examples \ref{eg:Toeplitz_representation_twist_dilation} and \ref{eg:disk_algebra_twist_dilation}, one might expect that the $C^*$-covers obtained in Theorem \ref{thm:uncountable_lattice} form an anti-chain, i.e. are pairwise non-comparable. However, the next example shows that they can even form a decreasing chain.

\begin{example}\label{eg:T2_maximal_representation_twist}
    By \cite[Example 2.4]{Blech}, the maximal $C^*$-cover of $T_2$ is given by 
      \begin{equation*}
        \{f\in C([0,1],M_2) \mid f(0) \textup{ is diagonal}\}
      \end{equation*}
    with the embedding
      \begin{equation*}
        i_{max}:\left(\begin{pmatrix} 
        a& b\\ 
        0 &d 
        \end{pmatrix}\right)\mapsto\begin{pmatrix} 
        a & b\sqrt{t}\\ 
        0 & c 
        \end{pmatrix},
      \end{equation*}
  where $t:[0,1]\to[0,1], x\mapsto x$. 
  
  A maximal dilation of $i_{max}$ is given by
  \begin{equation*}
    \begin{pmatrix} 
    a& b\\ 
    0 &d 
    \end{pmatrix} \mapsto\begin{pmatrix} 
    a & 0  & b\sqrt{1-t}  & -b\sqrt{t}           \\
    0  & a & b\sqrt{t}             & b\sqrt{1-t} \\ 
    0  & 0  & c             & 0             \\
    0  & 0  & 0              & c
    \end{pmatrix},
  \end{equation*}
  since 
  \begin{equation*}
  \sigma\left(\begin{pmatrix} 
        \sqrt{1-t} & -\sqrt{t}\\ 
        \sqrt{t} & \sqrt{1-t} 
        \end{pmatrix}\right)\subset\mathbb{T}.
  \end{equation*}
  For $s\in[0,1]$, define
 \begin{equation*}
  A_s=\begin{pmatrix} 
        s\sqrt{1-t} & -s^2\sqrt{t}\\ 
        \sqrt{t} & s\sqrt{1-t} 
        \end{pmatrix}.
  \end{equation*}
  The u.c.c. homomorphisms produced in Lemma \ref{lem:twist_dilation} are given by
  \begin{equation*}
    \begin{pmatrix} 
    a& b\\ 
    0 &d 
    \end{pmatrix}\mapsto\begin{pmatrix} 
    a1 & bA_s \\ 
    0   & c1
    \end{pmatrix}.
  \end{equation*}
Unitary conjugation shows that the $C^*$-cover isomorphism class induced by this u.c.c. homomorphism solely depends on the spectrum of $|A_s|$ if $A_s$ is invertible, and $\sigma(|A_s|)\cup\{0\}$ if $A_s$ is not invertible. It holds that
  \begin{equation*}
      A_s^*A_s=\begin{pmatrix} 
    s^2(1-t)+t   & (1-s^2)s\sqrt{t}\sqrt{1-t}\\ 
    (1-s^2)s\sqrt{t}\sqrt{1-t} & s^4t+s^2(1-t)
    \end{pmatrix}=s^21+(1-s^2)\sqrt{t}\begin{pmatrix} 
    \sqrt{t}             & s\sqrt{1-t}\\ 
    s\sqrt{1-t} & -s^2\sqrt{t}
    \end{pmatrix},
  \end{equation*}
and
 \begin{equation*}
      \sigma\left(\begin{pmatrix} 
    \sqrt{t}             & s\sqrt{1-t}\\ 
    s\sqrt{1-t} & -s^2\sqrt{t} 
    \end{pmatrix}\right)=\{\tfrac{1}{2}(\sqrt{x}(1-s^2)\pm\sqrt{x(s^2-1)^2+4s^2}) \mid x\in[0,1]\}.
  \end{equation*}
Hence
\begin{equation*}
      \sigma(|A_s|)=\left\{\sqrt{s^2+\tfrac{1}{2}(1-s^2)(x(1-s^2)\pm \sqrt{x}\sqrt{x(s^2-1)^2+4s^2})} \mid x\in[0,1]\right\}.
  \end{equation*}
Since the functions on the right side are continuous and agree for $x=0$, we obtain that the spectrum is an interval, and evaluating at $x=1$ shows that
\begin{equation*}
      [s^2,1]\subset\sigma(|A_s|).
  \end{equation*}
To see that we already have equality, observe that
 \begin{equation}
  A_s=\begin{pmatrix} 
        s & 0\\ 
        0 & 1
        \end{pmatrix}\begin{pmatrix} 
        \sqrt{1-t} & -\sqrt{t}\\ 
        \sqrt{t}            & \sqrt{1-t} 
        \end{pmatrix}\begin{pmatrix} 
        1 & 0\\ 
        0 & s
        \end{pmatrix}.
        \label{SpectrumMatrix}
  \end{equation}
Since the operator in the middle is unitary, we obtain that
 \begin{equation*}
      s^41\le A_s^*A_s,
 \end{equation*}
and therefore 
\begin{equation*}
      [s^2,1]=\sigma(|A_s|).
  \end{equation*}
So $\sigma(|A_s|)\subset\sigma(|A_u|)$ for $0\le u<s\le1$, and, by \eqref{SpectrumMatrix}, $A_s$ is invertible for every $s\in(0,1]$. Thus the $C^*$-covers obtained in Theorem \ref{thm:uncountable_lattice} form a decreasing chain.
\end{example}

This observation raises the following question:
\begin{question}
If $A$ is an operator algebra with more than one C*-cover, then is the C*-cover lattice of $A$ necessarily not totally ordered?
\end{question}

Note that in Example \ref{eg:T2_maximal_representation_twist}, the dilations formed via the process of Lemma \ref{lem:twist_dilation} were totally ordered. However, it is not true that the \emph{whole} C*-cover lattice of $T_2$ is totally ordered. Indeed, the Shilov ideal in $C^\ast_{\max}(T_2) = \{f \in C([0,1],M_2) \mid f(0) \text{ diagonal}\}$ is
\[
C_0([0,1),M_2)=\{f \in C([0,1],M_2) \mid f(1)=0\}.
\]
This has many non-totally ordered sub-ideals, which are boundary ideals for $T_2$ corresponding to the C*-covers of $T_2$, which are therefore not totally ordered.

\begin{remark}
One can also use the theory developed in this section to prove that any element in the Shilov ideal in $C^*_{\max}(A)$ for $A$ has to have a connected spectrum. Indeed, let $x$ be in the Shilov ideal and assume that $\sigma(x)$ equals the disjoint union of closed sets $K_1$ and $K_2$, both non-empty. Since the ideal does not contain $1$, we have $0\in\sigma(x)$ an without loss of generality we may assume that $0\in K_1$.

Then the proof of Theorem \ref{thm:uncountable_lattice} shows that there exists a Hilbert space $H$ and a family of unital homomorphisms $\pi_t:C^*_{\max}(A)\to\mathcal{B}(H)$, $t\in[0,1]$, such that $\pi_1(x)=0$, $\pi_0$ is injective, and 
  \[
    t\mapsto \|\pi_t(y)\|
  \]
is continuous for every $y\in C^*_{\max}(A)$.

Let $U_1, U_2$ be open disjoint sets in $\mathbb{C}$ such that $K_i\subset U_i, i=1,2$. Let $f:U_1\cup U_2\to \bbC$ be the holomorphic function with $f(U_1)=0$ and $f(U_2)=1$. Then
  \[
    f(\pi_t(x))=\pi_t(f(x))
  \]
for all $t\in[0,1]$ and, in particular,
  \[
    \pi_1(f(x))=f(\pi_1(x))=f(0)=0.
  \]
Note that we have $f(x)^2=f(x)$, and hence, for all $t\in[0,1]$, $\pi_t(f(x))^2=\pi_t(f(x))$. Moreover, if $\|\pi_t(f(x))\|<1$, then $\pi_t(f(x))=0$. Thus 
  \[
    \{t\in[0,1]:\ \|\pi_t(f(x))\|<1\}=\{t\in[0,1]:\ \pi_t(f(x))=0\},
  \]
and in particular these sets are open and closed. Since they contain $1$, we see that
  \[
    [0,1]=\{t\in[0,1];\ \pi_t(f(x))=0\},
  \]
which leads to the contradiction
  \[
    \{0,1\}=\sigma(f(x))=\sigma(\pi_0(f(x)))=\{0\}.
  \]
  
\end{remark}

\section{An example with no immediate successor to the C*-envelope}\label{Section:No successor to envelope}

For any C*-algebra $\C$ consider $T_{2}(\C)$ the $2\times 2$ upper triangular matrices with entries in $\C$.
It is immediate that $T_2(\C)$ is a Dirichlet operator algebra since $T_2(\C) + T_2(\C)^* = M_2(\C)$. By \cite[Lemma 5.2]{KatRamMem} the only Dirichlet representation is the C*-envelope, and so $C^*_e(T_2(\C)) = M_2(\C)$. 

Suppose $\varphi$ is a completely isometric representation of $T_2(\C)$. Then it must be of the form
\[
\varphi\left(\left[\begin{matrix} a & b \\ 0 & d\end{matrix}\right]\right) \ \ = \ \ \left[\begin{matrix} \varphi_{11}(a) & \varphi_{12}(b) \\ 0 & \varphi_{22}(d)\end{matrix}\right]
\]
where $\varphi_{11}, \varphi_{22}$ are $*$-isomorphisms of $\C$ and $\varphi_{12}$ is a completely isometric linear map of $\C$ with
\[
\varphi_{11}(a)\varphi_{12}(b)\varphi_{22}(d) = \varphi_{12}(abd)\,.
\]

\begin{proposition}\label{prop:ideal}
Suppose that $\C$ is a C*-algebra, and $\varphi$ is a completely isometric representation of $T_2(\C)$. 
If $\pi : C^*(\varphi(T_2(\C))) \rightarrow M_2(\C)$ is the unique surjective $*$-homomorphism such that $\pi(\varphi(A)) = A$, then $\ker \pi = \I$ where
\[
\I = \Big\langle (\varphi_{12}(c_1)^*\varphi_{12}(c_2) - \varphi_{22}(c_1^*c_2))E_{22}, \ (\varphi_{12}(c_1)\varphi_{12}(c_2)^* - \varphi_{11}(c_1c_2^*))E_{11} \ : \ c_1,c_2\in\C \Big\rangle\,.
\]
\end{proposition}
\begin{proof}
Certainly, for all $c_1,c_2 \in \C$ we have that
\begin{align*}
\pi\Big((\varphi_{12}(c_1)^*\varphi_{12}(c_2) - \varphi_{22}(c_1^*c_2))E_{22}\Big)
& = \pi\circ \varphi\left(\left[\begin{matrix} 0 & c_1 \\ 0 & 0\end{matrix}\right] \right)^*\pi\circ \varphi\left(\left[\begin{matrix} 0 & c_2 \\ 0 & 0\end{matrix}\right] \right) - \pi\circ \varphi\left(\left[\begin{matrix} 0 & 0 \\ 0 & c_1^*c_2\end{matrix}\right] \right)
\\ & = \left[\begin{matrix} 0 & 0 \\ c_1^* & 0\end{matrix}\right]\left[\begin{matrix} 0 & c_2 \\ 0 & 0\end{matrix}\right] - \left[\begin{matrix} 0 & 0 \\ 0 & c_1^*c_2\end{matrix}\right] = 0
\end{align*}
and $\pi\Big((\varphi_{12}(c_1)\varphi_{12}(c_2)^* - \varphi_{11}(c_1c_2^*))E_{11} \Big) = 0$ similarly.
Thus, $\I \subseteq \ker \pi$.

Now, let $\pi_\I : C^*(\varphi(T_2(\C))) \rightarrow C^*(\varphi(T_2(\C)))/\I$ be the quotient map. Then $\pi_\I \circ \varphi$ is still a completely isometric representation of $T_2(\C)$. Now we have that
\begin{align*}
& \pi_\I\circ\varphi\left(\left[\begin{matrix} a_1 & b_1 \\ 0 & d_1\end{matrix}\right]\right)^*\pi_\I\circ\varphi\left(\left[\begin{matrix} a_2 & b_2 \\ 0 & d_2\end{matrix}\right]\right)
\\& \quad \quad = \pi_\I\left(\left[\begin{matrix} \varphi_{11}(a_1^*) & 0 \\ \varphi_{12}(b_1)^* & \varphi_{22}(d_1^*)\end{matrix}\right]\left[\begin{matrix} \varphi_{11}(a_2) & \varphi_{12}(b_2) \\ 0 & \varphi_{22}(d_2)\end{matrix}\right] \right)
\\ & \quad \quad = \pi_\I\left( \left[\begin{matrix} \varphi_{11}(a_1^*a_2) & \varphi_{12}(a_1^*b_2) \\ \varphi_{12}(a_2^*b_1)^* & \varphi_{12}(b_1)^*\varphi_{12}(b_2) + \varphi_{22}(d_1^*d_2)\end{matrix}\right]\right)
\\ & \quad \quad = \pi_\I\left( \left[\begin{matrix} \varphi_{11}(a_1^*a_2) & \varphi_{12}(a_1^*b_2) \\ \varphi_{12}(a_2^*b_1)^* & \varphi_{22}(b_1^*b_2 + d_1^*d_2)\end{matrix}\right]\right) + \pi_\I((\varphi_{12}(b_1)^*\varphi_{12}(b_2) - \varphi_{22}(b_1^*b_2))E_{22})
\\ & \quad \quad = \pi_\I\circ\varphi\left( \left[\begin{matrix} a_1^*a_2 & a_1^*b_2 \\ 0 & b_1^*b_2 + d_1^*d_2\end{matrix}\right]\right) + \pi_\I\circ\varphi\left( \left[\begin{matrix}0 & a_2^*b_1 \\ 0 & 0\end{matrix}\right]\right)^*
\\ & \quad \quad \in \pi_\I\circ\varphi(T_2(\C)) + \pi_\I\circ\varphi(T_2(\C))^*\,.
\end{align*}
Hence, $\pi_\I\circ\varphi$ is a semi-Dirichlet completely isometric representation. In exactly the same way one can prove that $\pi_\I\circ\varphi$ is also a semi-Dirichlet$^*$ representation. 
Thus, $\pi_\I\circ\varphi$ is a Dirichlet representation, or equivalently that $\overline{\pi_\I\circ\varphi(T_2(\C)) + \pi_\I\circ\varphi(T_2(\C))^*}$ is a C*-algebra. 

Therefore, since as mentioned above there is only one unique Dirichlet representation, $C^*(\pi_\I\circ\varphi(T_2(\C)))$ and $M_2(\C)$ are isomorphic as C$^*$-covers and so $\ker \pi = \ker \pi_\I = \I$.
\end{proof}

Gramsch \cite{Gramsch} and Luft \cite{Luft} independently identified the ideal structure of $\B(H)$ in the non-separable case. To this end, for any cardinal number $\omega$, define $t\in \B(H)$ to be \textbf{$\omega$-compact} if each closed linear subspace of $t(H)$ has dimension strictly smaller than $\omega$. Denote the collection of all $\omega$-compact operators of $\B(H)$ by $K_\omega(H)$, which is a closed two-sided ideal of $\B(H)$. Gramsch and Luft proved that these are the only closed two-sided ideals and so they form a chain.

Consider now the operator algebra $A = T_2(K_\omega(H))$ where $\dim(H) = \omega$ is a limit cardinal strictly bigger than $\aleph_0$. Note that $\B(H) = K_{\omega+1}(H)$. The significant property of $K_\omega(H)$ is that it has no maximal closed ideals.

\begin{proposition}
Suppose $\varphi$ is a completely isometric representation of $T_2(K_\omega(H))$ that does not generate the C$^*$-envelope as a C$^*$-cover. 
If $\pi : C^*(\varphi(T_2(K_\omega(H)))) \rightarrow M_2(K_\omega(H))$ is the unique surjective $*$-homomorphism such that $\pi(\varphi(A)) = A$ and a cardinal $\aleph_0 < \mu < \omega$, then 
\[
\I_\mu = \Big\langle (\varphi_{12}(p)^*\varphi_{12}(p) - \varphi_{22}(p))E_{22}, \ (\varphi_{12}(p)\varphi_{12}(p)^* - \varphi_{11}(p))E_{11} \ : \ p\in K_\mu(H) \textrm{ projection} \Big\rangle
\]
is an ideal of, but not equal to, $\ker \pi$.
\end{proposition}
\begin{proof}
By the previous proposition $\I_\mu \subseteq \ker \pi$ which means that it is certainly an ideal of $\ker \pi$. If $\I_\mu = \{0\}$ then the conclusion follows.

Assume now that $\I_\mu \neq \{0\}$. Therefore, there exists a projection $p\in K_\mu(H)$ such that 
\[
\varphi\left(\left[\begin{matrix} 0 & p \\ 0 & -p\end{matrix}\right]\right)^*\varphi\left(\left[\begin{matrix} 0 & p \\ 0 & p\end{matrix}\right]\right) \neq 0 \quad \textrm{or} \quad\varphi\left(\left[\begin{matrix} p & -p \\ 0 & 0\end{matrix}\right]\right)\varphi\left(\left[\begin{matrix} p & p \\ 0 & 0\end{matrix}\right]\right)^* \neq 0\,.
\]
Without loss of generality, assume the former is non-zero. Suppose $q\in K_\mu(H)$ is any other projection equivalent to $p$, so that there is a partial isometry $v$ satisfying $v^*v = p$ and $vv^*=q$. Then notice that
\begin{align*}
0 & \neq \varphi\left(\left[\begin{matrix} 0 & p \\ 0 & -p\end{matrix}\right]\right)^*\varphi\left(\left[\begin{matrix} 0 & p \\ 0 & p\end{matrix}\right]\right)
\\ & = \varphi\left(\left[\begin{matrix} 0 & v^*qv \\ 0 & -v^*qv\end{matrix}\right]\right)^*\varphi\left(\left[\begin{matrix} 0 & v^*qv \\ 0 & v^*qv\end{matrix}\right]\right)
\\ & = \varphi\left(\left[\begin{matrix} v^* & 0 \\ 0 & v^*\end{matrix}\right]\right)\varphi\left(\left[\begin{matrix} 0 & q \\ 0 & -q\end{matrix}\right]\right)^*\varphi\left(\left[\begin{matrix} vv^* & 0 \\ 0 & vv^*\end{matrix}\right]\right)\varphi\left(\left[\begin{matrix} 0 & q \\ 0 & q\end{matrix}\right]\right)\varphi\left(\left[\begin{matrix} v & 0 \\ 0 & v\end{matrix}\right]\right)
\\ & = \varphi\left(\left[\begin{matrix} v^* & 0 \\ 0 & v^*\end{matrix}\right]\right)\varphi\left(\left[\begin{matrix} 0 & q \\ 0 & -q\end{matrix}\right]\right)^*\varphi\left(\left[\begin{matrix} 0 & q \\ 0 & q\end{matrix}\right]\right)\varphi\left(\left[\begin{matrix} v & 0 \\ 0 & v\end{matrix}\right]\right)\,.
\end{align*}
In particular, this implies that 
\[
\varphi\left(\left[\begin{matrix} 0 & q \\ 0 & -q\end{matrix}\right]\right)^*\varphi\left(\left[\begin{matrix} 0 & q \\ 0 & q\end{matrix}\right]\right) \neq 0
\]

Now for a projection $r \in K_\omega(H) \setminus K_\mu(H)$ consider the element
\[
\varphi\left(\left[\begin{matrix} 0 & r \\ 0 & -r\end{matrix}\right]\right)^*\varphi\left(\left[\begin{matrix} 0 & r \\ 0 & r\end{matrix}\right]\right)\,,
\]
Because of the way $\I_\mu$ is constructed and since $K_\mu(H)$ is an ideal then
\[
\I_\mu \subseteq M_2(C^*(\{\varphi_{11}(K_\mu(H)), \varphi_{12}(K_\mu(H)), \varphi_{22}(K_\mu(H))\}))\,.
\]
Thus, for any $a\in \I_\mu$ by cardinality, and since $\aleph_0<\mu$, there must be a subprojection $q$ of $r$ that is equivalent to $p$ such that 
\[
\varphi\left(\left[\begin{matrix} q & 0 \\ 0 & q\end{matrix}\right]\right)a = a\varphi\left(\left[\begin{matrix} q & 0 \\ 0 & q\end{matrix}\right]\right) = 0\,.
\]
Hence,
\begin{align*}
& \left\|\varphi\left(\left[\begin{matrix} 0 & r \\ 0 & -r\end{matrix}\right]\right)^*\varphi\left(\left[\begin{matrix} 0 & r \\ 0 & r\end{matrix}\right]\right) - a\right\|
\\ & \quad \quad \quad \quad \quad \quad \geq \left\|\varphi\left(\left[\begin{matrix} q & 0 \\ 0 & q\end{matrix}\right]\right)^*\left(\varphi\left(\left[\begin{matrix} 0 & r \\ 0 & -r\end{matrix}\right]\right)^*\varphi\left(\left[\begin{matrix} 0 & r \\ 0 & r\end{matrix}\right]\right) - a\right)\varphi\left(\left[\begin{matrix} q & 0 \\ 0 & q\end{matrix}\right]\right)\right\| 
\\ & \quad \quad \quad \quad \quad \quad = \left\|\varphi\left(\left[\begin{matrix} 0 & q \\ 0 & -q\end{matrix}\right]\right)^*\varphi\left(\left[\begin{matrix} 0 & q \\ 0 & q\end{matrix}\right]\right)\right\|
\\ & \quad \quad \quad \quad \quad \quad \geq \left\|\varphi\left(\left[\begin{matrix} 0 & p \\ 0 & -p\end{matrix}\right]\right)^*\varphi\left(\left[\begin{matrix} 0 & p \\ 0 & p\end{matrix}\right]\right)\right\| > 0\,.
\end{align*}
Therefore, $\I_\mu\neq \ker \pi$.
\end{proof}

Now we can prove the main theorem of this section.

\begin{theorem}
There is no immediate successor to $M_2(K_\omega(H))$ in $\cstarlattice(T_2(K_\omega(H))$.
\end{theorem}
\begin{proof}
Suppose $\varphi$ is a completely isometric representation of $T_2(K_\omega(H))$ that does not generate the C$^*$-envelope as a C$^*$-cover. As before, let $\pi: C^*(\varphi(T_2(K_\omega(H)))) \rightarrow M_2(K_\omega(H))$ be the canonical surjective $*$-homomorphism such that $\pi\circ\varphi(A) = A$ for all $A\in T_2(K_\omega(H))$.

Now because $\ker \pi \neq \{0\}$ by Proposition \ref{prop:ideal} there exist $c_1,c_2\in K_\omega(H)$ such that
\[
\varphi\left(\left[\begin{matrix} 0 & c_1 \\ 0 & -c_1\end{matrix}\right]\right)^*\varphi\left(\left[\begin{matrix} 0 & c_2 \\ 0 & c_2\end{matrix}\right]\right) \neq 0 \quad \textrm{or} \quad\varphi\left(\left[\begin{matrix} c_1 & -c_1 \\ 0 & 0\end{matrix}\right]\right)\varphi\left(\left[\begin{matrix} c_2 & c_2 \\ 0 & 0\end{matrix}\right]\right)^* \neq 0\,.
\]
Suppose the first is non-zero and $\{p_\lambda\}_\lambda$ is a approximate unit of $K_\omega(H)$ consisting of projections. Thus,
\begin{align*}
\varphi\left(\left[\begin{matrix} 0 & c_1 \\ 0 & -c_1\end{matrix}\right]\right)^*\varphi\left(\left[\begin{matrix} 0 & c_2 \\ 0 & c_2\end{matrix}\right]\right)
& = \lim_\lambda \varphi\left(\left[\begin{matrix} 0 & p_\lambda c_1 \\ 0 & -p_\lambda c_1\end{matrix}\right]\right)^*\varphi\left(\left[\begin{matrix} 0 & p_\lambda c_2 \\ 0 & p_\lambda c_2\end{matrix}\right]\right)
\\ & = \lim_\lambda \varphi\left(\left[\begin{matrix} c_1^* & 0  \\ 0 & c_1^* \end{matrix}\right]\right)\varphi\left(\left[\begin{matrix} 0 & p_\lambda \\ 0 & -p_\lambda \end{matrix}\right]\right)^*\varphi\left(\left[\begin{matrix} 0 & p_\lambda  \\ 0 & p_\lambda \end{matrix}\right]\right)\varphi\left(\left[\begin{matrix} c_2 & 0  \\ 0 & c_2 \end{matrix}\right]\right)\,.
\end{align*}
In particular, there is some projection $p$ and cardinal $\mu<\omega$ such that $p\in \I_\mu$ and 
\[
\varphi\left(\left[\begin{matrix} 0 & p \\ 0 & -p\end{matrix}\right]\right)^*\varphi\left(\left[\begin{matrix} 0 & p  \\ 0 & p\end{matrix}\right]\right) \neq 0\,.
\]

By the previous proposition this implies that $I_\mu$ is a proper ideal of $\ker \pi$. Therefore, $C^*(T_2(K_\omega(H)))/I_\mu$ is a C$^*$-cover sitting strictly between $C^*(T_2(K_\omega(H)))$ and $M_2(K_\omega(H))$.
\end{proof}

\begin{remark}
There are also examples of separable, non-simple C*-algebras with no maximal closed ideals, see \cite{Leptin}. Hence, it seems likely that one could find a separable operator algebra which has no immediate successor to its C*-envelope in the lattice of C*-covers.
\end{remark}

\section{Semi-Dirichlet C*-covers}\label{Section:Semi-Dirichlet C*-covers}

The following argument answers \cite[Question 3.13]{HKR} which asks when $\cmax(A)$ is a semi-Dirichlet C*-cover. Obviously, there is some background material that needs to be added here. However, it does fit well with our current results since it is using the same trick of scaling the off-diagonal entry to generate new C*-covers. 

As defined in \cite{HKR}, a representation $\rho: A \rightarrow \B(H)$ is semi-Dirichlet if $\rho(A)^*\rho(A) \subseteq \overline{\rho(A) + \rho(A)^*}$. A C*-cover $[\C,\iota]$ is called semi-Dirichlet if $\iota$ is a semi-Dirichlet representation. An operator algebra $A$ is called semi-Dirichlet if and only if it has a semi-Dirichlet C*-cover.

\begin{theorem}\label{thm:semidirichlet}
    If $A$ is a non-selfadjoint, semi-Dirichlet operator algebra, then the maximal C*-cover is not a semi-Dirichlet C*-cover.
\end{theorem}
\begin{proof}

By contradiction, assume that $[\cmax(A),\iota_{\max}]$ is semi-Dirichlet. Suppose $\varphi:A \rightarrow \B(H)$ is any completely contractive representation. By universality, there exists $\pi: \cmax(A) \rightarrow C^*(\rho(A))$ such that $\pi\circ \iota_{\max} = \rho$. Then $\rho$ must be a semi-Dirichlet representation:
\begin{align*}
    \rho(a)^*\rho(b) & = \pi(\iota_{\max}(a)^*\iota_{\max}(b))
    \\ & = \pi\left( \lim_{n\rightarrow \infty} \iota_{\max}(c_n) + \iota_{\max}(d_n)^*\right)
    \\ & = \lim_{n\rightarrow \infty} \rho(c_n) + \rho(d_n)^*\,.
\end{align*}
We will see that it is impossible that all representations are semi-Dirichlet.

Now, by \cite[Proposition 3.4]{HumRam}, since $A$ is non-selfadjoint, there exists a non-maximal, completely contractive representation $\varphi : A \rightarrow \B(H)$.
Non-maximality implies that $\varphi$ either extends or coextends non-trivially.

First, assume that $\varphi$ extends non-trivially to a completely contractive representation $\Phi : A \rightarrow \B(K)$ with $K = H \oplus H^\perp$ and block structure
\[
\Phi = \left[\begin{matrix} \varphi & \Phi_{12} \\ 0& \Phi_{22} \end{matrix}\right]
\]
where $\Phi_{12} \neq 0$.
Define $\Phi' : A \rightarrow \B(K)$ 
\[
\Phi' = \left[\begin{matrix} \varphi & \frac{1}{2}\Phi_{12} \\ 0& \Phi_{22} \end{matrix}\right]\,.
\]
Both $\Phi$ and $\Phi'$ are semi-Dirichlet representations by assumption.

Now, there exists $a\in A$ such that $\Phi_{12}(a)\neq 0$. By the semi-Dirichlet property for $\iota_{\max}$ there exist $b_n\in A$ such that
\begin{align*}
\iota_{\max}(a)^*\iota_{\max}(a) &= \lim_{n\rightarrow \infty} \iota_{\max}(b_n) + \iota_{\max}(b_n)^*\,.
\end{align*}
By universality, there exists $*$-homomorphisms $\pi,\pi':\cmax(A) \rightarrow \B(K)$ such that $\pi\circ\iota_{\max} = \Phi$ and $\pi'\circ\iota_{\max} = \Phi'$.
Following the argument at the start of this proof,
\begin{align*}
\Phi(a)^*\Phi(a) &= \lim_{n\rightarrow \infty} \Phi(b_n) + \Phi(b_n)^*, \quad \textrm{and}
\\ \Phi'(a)^*\Phi'(a) &= \lim_{n\rightarrow \infty} \Phi'(b_n) + \Phi'(b_n)^*\,.
\end{align*}
But this implies that
\begin{align*}
    \left[\begin{matrix} \varphi(a)^*\varphi(a) & \varphi(a)^*\Phi_{12}(a) \\ \Phi_{12}(a)^*\varphi(a) & \Phi_{12}(a)^*\Phi_{12}(a) + \Phi_{22}(a)^*\Phi_{22}(a)\end{matrix}\right]
    &= \lim_{n\rightarrow \infty} \left[\begin{matrix} \varphi(b_n) + \varphi(b_n)^* & \Phi_{12}(b_n)  \\ \Phi_{12}(b_n)^* & \Phi_{22}(b_n) + \Phi_{22}(b_n)^* \end{matrix}\right]
\end{align*}
and 
\begin{align*}
    \left[\begin{matrix} \varphi(a)^*\varphi(a) & \frac{1}{2}\varphi(a)^*\Phi_{12}(a) \\ \frac{1}{2}\Phi_{12}(a)^*\varphi(a) & \frac{1}{4}\Phi_{12}(a)^*\Phi_{12}(a) + \Phi_{22}(a)^*\Phi_{22}(a)\end{matrix}\right]
    &= \lim_{n\rightarrow \infty} \left[\begin{matrix} \varphi(b_n) + \varphi(b_n)^* & \frac{1}{2}\Phi_{12}(b_n)  \\ \frac{1}{2}\Phi_{12}(b_n)^* & \Phi_{22}(b_n) + \Phi_{22}(b_n)^* \end{matrix}\right]\,.
\end{align*}
Looking at the (2,2)-entries of both equations we see that $\frac{3}{4}\Phi_{12}(a)^*\Phi_{12}(a) = 0$ and so $\Phi_{12}(a) = 0$, a contradiction.

Essentially the same argument works in the case that $\varphi$ has a nontrivial coextension, in which case the operator matrices involved are lower triangular, so we will omit those details.

Therefore, it cannot be the case that both $\Phi$ and $\Phi'$ are semi-Dirichlet representations, which implies that $[\cmax(A),\iota_{\max}]$ is not a semi-Dirichlet C*-cover.
\end{proof}

\begin{corollary}
If $A$ is a non-selfadjoint, semi-Dirichlet operator algebra, then it has an uncountable lattice of C*-covers.
\end{corollary}
\begin{proof}
By definition $[C^*_e(A),\iota_e]$ is a semi-Dirichlet cover and $[\cmax(A),\iota_{\max}]$ is not a semi-Dirichlet cover. Since there are two points in the lattice then there are an uncountable number by Theorem \ref{thm:uncountable_lattice}.
\end{proof}

\section*{Acknowledgements}
The first author was supported by the NSERC Discovery Grant 2024-03883. The second author was supported by the NSERC Discovery Grant 2019-05430. The third author was supported by the Emmy Noether Program of the German Research Foundation (DFG Grant 466012782).


\begin{thebibliography}{99}

 \bibitem{Arv1} W. Arveson, \textit{Subalgebras of C$^*$-algebras}, Acta Mathematica {\bf 123} (1969), no. 1, 141–224.

 \bibitem{Arv2} W. Arveson, \textit{An Invitation to C$^\ast$-Algebras}, Graduate Texts in Mathematics, Vol.~39, Springer-Verlag, New York, 1976.

 \bibitem{Blech} D. Blecher, \textit{Modules over operator algebras and the maximal C$^*$-dilation}, J. Func. Anal. {\bf 169} (1999), 251-288.

 \bibitem{Blecher} D. Blecher and C. Le Merdy, \textit{Operator algebras and their modules: an operator space approach}, Oxford University Press (2004), 30.


\bibitem{CourtneySherman}
K. Courtney and D. Sherman. \textit{The universal C*-algebra of a contraction}, J. Operator Theory {\bf 84} (2020), 153-184.

\bibitem{Dav}
K.~R. Davidson, {\it $C^*$-algebras by example}, Fields Institute Monographs, 6, Amer. Math. Soc., Providence, RI, 1996; MR1402012


\bibitem{DriMcc}
M. Dritschel and S. McCullough, \textit{Boundary representations for families of representations of operator algebras and spaces}, J. Op. Thy. (2005), 159–167.

\bibitem{Gramsch}
B. Gramsch, {\it Eine Idealstruktur Banachscher Operatoralgebren}, J. Reine Angew. Math. {\bf 225} (1967), 97–115.


 \bibitem{Hamana}
 M. Hamana, \textit{Injective envelopes of operator systems}, Publications of the Research
 Institute for Mathematical Sciences {\bf 15} (1979), no. 3, 773–785.

\bibitem{Hamidi}
M. Hamidi, \textit{Admissibility of C$^*$-covers and crossed products of operator algebras}, Ph.D. Thesis, \url{https://digitalcommons.unl.edu/mathstudent/95} (2019).


\bibitem{HKR}
     A. Humeniuk, E. Katsoulis, and C. Ramsey, \textit{Crossed products and C*-covers of semi-Dirichlet operator algebras}, Doc. Math. {\bf 30} (2025), 909-933.

\bibitem{HumRam}
A. Humeniuk and C. Ramsey, \textit{The lattice of C$^*$-covers of an operator algebra}, Canad. J. Math. (2025), pp. 1-29.

\bibitem{HumRamThom}
A. Humeniuk, C. Ramsey, and I. Thompson, \textit{Couniversality for C*-algebras of residually finite-dimensional operator algebras}, preprint arXiv:2507.11824 (2025).

\bibitem{Kacnelson} 
      V.~E. Katsnelson, \textit{A remark on canonical factorization in certain spaces of analytic functions}, Zap. Nau\v cn. Sem. Leningrad. Otdel. Mat. Inst. Steklov. (LOMI) {\bf 30} (1972), 163--164; MR0338819.

\bibitem{KatRamMem} E. Katsoulis and C. Ramsey, \textit{Crossed products of operator algebras}, Mem. Amer. Math. Soc {\bf 258} (2019), no. 1240, vii+85 pp.

\bibitem{KirWas}
E. Kirchberg and S. Wassermann, \textit{C$^*$-algebras generated by operator systems}, J. Func. Anal. {\bf 155} (1998), 324-351.

\bibitem{Leptin}
H. Leptin, \textit{A separable postliminal C*-algebra without maximal closed ideals}, Trans. Amer. Math. Soc. {\bf 159} (1971), 489-496. 


\bibitem{Luft}
E. Luft, \textit{The two-sided closed ideals of the algebra of bounded linear operators of a Hilbert space}, Czech. Math. J. {\bf 18} (1968), 595–605.

\bibitem{Murphy}
G.~J. Murphy, {\it $C^*$-algebras and operator theory}, Academic Press, Boston, MA, 1990; MR1074574


\bibitem{Sarason}
      D. Sarason, \emph{On spectral sets having connected complement}, Acta Sci. Math. (Szeged), \textbf{26}, 289-299.

\bibitem{Thompson} I. Thompson, \textit{Maximal C$^*$-covers and residual finite-dimensionality}, J. Math. Anal. App. {\bf 514} (2022), 126277.


\bibitem{Willard} S. Willard, \textit{General topology}, Reprint of the 1970 original [Addison-Wesley, Reading, MA; MR0264581], 2004, xii+369, MR2048350.


      
\end{thebibliography}
\end{document}